\documentclass[reqno]{amsart}

\usepackage{amsmath,amssymb,amsfonts,amsthm}
\usepackage[dvipsnames]{xcolor}
\usepackage{enumerate}
\usepackage{todonotes}
\usepackage{caption}
\usepackage{subcaption}

\DeclareFontFamily{OT1}{pzc}{}
\DeclareFontShape{OT1}{pzc}{m}{it}{<-> s * [0.900] pzcmi7t}{}
\DeclareMathAlphabet{\mathscr}{OT1}{pzc}{m}{it}

\usepackage{graphicx}
\usepackage{grffile}

\usepackage[utf8]{inputenc}

\usepackage[nonotes]{optional}

\usepackage[sorting=none,natbib=true,style=numeric,backend=biber]{biblatex}

\newcommand{\Figref}[1]{Fig.~\ref{#1}}

\newcommand{\Thmref}[1]{Theorem~\ref{#1}}
\newcommand{\Corref}[1]{Corollary~\ref{#1}}
\newcommand{\Lemref}[1]{Lemma~\ref{#1}}
\newcommand{\Propref}[1]{Prop.~\ref{#1}}
\newcommand{\Defref}[1]{Definition~\ref{#1}}

\newcommand{\R}{\mathbb{R}} 
\newcommand{\Z}{\mathbb{Z}} 
\newcommand{\Q}{\mathbb{Q}} 
\newcommand{\Prob}{\mathbb{P}} 

\newcommand{\e}{\epsilon}
\newcommand{\w}{\omega}
\newcommand{\W}{\Omega}
\newcommand{\grad}{\nabla} 
\newcommand{\AND}{\textrm{ and }}

\newcommand{\Var}{\operatorname{Var}}
\newcommand{\Cov}{\operatorname{Cov}}
\newcommand{\almostsurely}{\textrm{a.s}}
\newcommand{\suchthat}{\textrm{s.t}}

\renewcommand{\liminf}{\varliminf}

\newcommand{\Exp}{\operatorname{Exp}}

\newcommand{\Uniform}{\operatorname{Uniform}}

\newcommand{\Bernoulli}{\operatorname{Bernoulli}}

\newcommand{\Lognormal}{\operatorname{Lognormal}}

\newcommand{\E}{\mathbb{E}}

\renewcommand{\emph}[1]{\textbf{#1}}

\theoremstyle{plain}
\newtheorem{theorem}{Theorem}[section]
\newtheorem*{theorem*}{Theorem}
\newtheorem{lemma}[theorem]{Lemma}
\newtheorem*{lemma*}{Lemma}
\newtheorem{cor}[theorem]{Corollary}

\newtheorem{prop}[theorem]{Proposition}
\newtheorem*{prop*}{Proposition}

\theoremstyle{definition}
\newtheorem{define}[theorem]{Definition}

\newtheorem{question}{Question}
\newtheorem*{question*}{Question}

\newtheorem*{example*}{Example}

\theoremstyle{remark}
\newtheorem{remark}{Remark}
\newtheorem*{remark*}{Remark}

\makeatletter
\@namedef{subjclassname@2020}{%
      \textup{2020} Mathematics Subject Classification}
  \makeatother

\title{Negative correlation of adjacent Busemann increments}
\author[I.~Alevy]{Ian Alevy}
\address{Ian Alevy\\ University of Rochester\\ Mathematics Department\\ Hylan 1008\\   Rochester, NY 14627\\ USA.}
\email{ian.alevy@rochester.edu}
\urladdr{https://people.math.rochester.edu/faculty/ialevy/}

\author[A.~Krishnan]{Arjun Krishnan}
\address{Arjun Krishnan\\ University of Rochester\\ Mathematics Department\\ Hylan 817\\   Rochester, NY 14627\\ USA.}
\email{arjunkc@gmail.com}
\urladdr{https://people.math.rochester.edu/faculty/akrish11/}

\date{February 10, 2021}



\def \fpp {first-passage percolation}

\def \lpt {last-passage time}

\def \lpp {last-passage percolation}

\def \limitshape {\mathcal{B}}

\def \derivatives {\mathcal{D}}
\def \grads {\mathcal{G}}
\def \indexset {\mathcal{K}}

\def \good {\operatorname{GOOD}}

\def \path {\operatorname{PATH}}

\def \pathp {\operatorname{PATH}^\prime}

\def \antidiaglines {\mathcal{L}}

\newcommand{\oldarjun}[1]{\opt{oldnotes}{\todo[size=\tiny,backgroundcolor=white,noline]{\textcolor{Mahogany}{arjun: #1}}}}

\newcommand{\gexp}[1][]{g_{\text{Exp}}}

\newcommand{\directions}{\mathcal{U^{\circ}}}

\usepackage[urlcolor=blue,colorlinks=true,linkcolor=blue,citecolor=blue]{hyperref}

\begin{document}

\begin{abstract}
We consider i.i.d.\ \lpp{} on $\Z^2$ with weights having distribution $F$ and time-constant $g_F$. We provide an explicit condition on the large deviation rate function for independent sums of $F$ that determines when some adjacent Busemann function increments are negatively correlated. As an example, we prove that $\Bernoulli(p)$ weights for $p > p^* \approx 0.6504$ satisfy this condition.  We prove this condition by establishing a direct relationship between the negative correlations of adjacent Busemann increments and the dominance of the time-constant $g_F$ by the function describing the time-constant of \lpp{} with exponential or geometric weights.
\end{abstract}

\maketitle

\setcounter{tocdepth}{1}
\tableofcontents

\section{Introduction}

\textbf{Directed last passage percolation} (LPP) is a growth model on a directed graph with random edge or vertex weights. In this paper we focus on the directed nearest-neighbor lattice graph \(\Z^2\), with non-negative i.i.d.\ vertex weights $\{\w_x\}_{x \in \Z^2}$. We say that \(x \prec  y\) if  $x_i \leq y_i$ for $i = 1,2$, in which case \(x\) and \(y\) can be connected by an \textbf{up/right path}: this is a sequence of vertices \(\Gamma = \{x=x_0,x_1, \ldots, x_k=y\}\) in which each step is either right or up; i.e., \(x_{i+1}-x_{i}\in\{e_1,e_2\}\), the canonical unit directions in $\Z^2$. The \textbf{passage-time} of \(\Gamma\) is 
$$
    W(\Gamma) = \sum_{x \in \Gamma} \w_x.
$$

The \textbf{last-passage time} from $x$ to $y$ \((x \prec y)\) is
\begin{align}\label{eq:passage time definition}
    G(x,y) = \sup_{\Gamma} W(\Gamma),
\end{align}
where the supremum is over all up/right paths from $x$ to $y$. A \textbf{geodesic} is an up/right path between points \(x\) and \(y\) that achieves the supremum in \eqref{eq:passage time definition}. An \textbf{infinite geodesic} is an up/right path that is a geodesic between any two points on it.

If $\w_x$ is in $L^1$, the classical subadditive ergodic theorem ensures the existence of the so-called \emph{time-constant} \cite{MR1258174,MR1241039}, which is the limit 
$$
    \lim_{n \to \infty} \frac{G(0,[nx])}{n} = g_F(x) \quad \almostsurely \AND \text{in } L^1
$$
for each $x \in \R_{\geq 0}^2$, where $[y]$ is the only lattice point in $[y,y+1)^2$. The time-constant is a 1-homogeneous, concave function which respects the symmetries of the lattice (e.g.\ $g(x,y) = g(y,x)$). If \(E[w_x^{2+\epsilon}]<\infty\) for some $\e > 0$, then \(g(x)\) is finite for all \(x\in \R^2_{\geq 0}\) \citep{MR1241039,MR1258174}; see also \cite{martin2002} for a slightly weaker sufficient condition. The \textbf{limit-shape} is the level set of the time-constant $\mathcal{B} = \{ x \colon g(x) = 1 \}$. For general i.i.d.\ weights, \(g(x)\) is poorly understood.

Last-passage percolation with exponential or geometrically distributed i.i.d.\ vertex weights are the only known \textit{integrable} or \textit{solvable} cases. Let the weights have mean \(m\) and variance $\sigma^2$. Then, the limit-shape for both exponential and geometric weights is given by \citep{rost_non-equilibrium_1981}
\begin{equation}
    g_{\Exp}(x,y) = m(x + y) + 2 \sigma \sqrt{ x y } \quad\forall x,y \in \R_{\geq 0}^2.
    \label{eq:exponential-lpp-limit-shape}
\end{equation}
When the weights are exponentials, of course, we must have $m = \sigma$, but we will use $g_{\Exp}(x,y)$ to simply mean the function on the right hand side of \eqref{eq:exponential-lpp-limit-shape} with parameters \(m\) and \(\sigma\). Martin \citep{MR2094434} showed that this shape is asymptotically universal close to the vertical and horizontal axes:
$$
    g_F(1,s) = m   + 2 \sigma \sqrt{ s } + o(\sqrt{s}) \text{ as } s \to 0. 
$$

In the solvable models, the random fluctuations of the \lpt{} are known to be in the KPZ universality class \citep{krug_universality_1988} for growth models, since \citet{johansson_shape_2000} proved that 
\begin{equation}
    \lim_{N \to \infty} \Prob\left( \frac{G(0,[Nx]) - N \gexp(x)}{c(x) N^{\chi}} \leq t \right) = F_{GUE}(t),
    \label{eq:gue limit for last passage time}
\end{equation}
where $\chi=1/3$, \(c(x)\) is an explicit function, and $F_{GUE}(t)$ is the cdf of the GUE Tracy-Widom distribution \cite{MR1257246}. Equation~\eqref{eq:gue limit for last passage time} is conjectured to be true for all ``nice enough'' i.i.d.\ weights \citep{johansson_shape_2000}; i.e., the Tracy-Widom distribution is a universal limit. The exponent $\chi$ is called the \emph{fluctuations exponent} of the passage time. 

There is another exponent closely associated with \(\chi\) called the \emph{geodesic wandering exponent} $\xi$. One way of defining it is as follows \cite{licea_superdiffusivity_1996}: Let $L_x$ be the straight line between $0$ and $x$, and let $C(\gamma,x) \subset \R^2$ be the cylinder with central axis \(L_x\), radius \(N^{\gamma}\) and length \(|x|_1\). Let $A_N^{\gamma}$ be the event that all geodesics from $0$ to $Nx$ are contained inside the cylinder $C(\gamma,Nx)$, and let
\begin{equation*}
    \xi = \inf \{ \gamma > 0 \colon \liminf_{N \to \infty} \Prob( A_N^{\gamma} ) = 1 \} .
\end{equation*}
Originally, \citet{MR1757595} proved that $\xi = 2/3$ in a related model of two dimensional growth, the Poissonized longest increasing subsequence problem.
In the solvable \lpp{} models, this was shown by \citet{MR2268539}. In these solvable models, since $\chi = 1/3$ and $\xi = 2/3$, the \emph{first KPZ scaling relationship} $\chi = 2 \xi  - 1$ holds. This scaling relation is again conjectured to be universal~\citep{krug_universality_1988}, in that it is supposed to hold for a large class of growth models with i.i.d.\ weights, and all dimensions \(d\geq 2\). In \fpp{}, this conjecture has been proven under various unproven hypotheses on the limit-shape and the existence of these exponents \cite{chatterjee_universal_2011,MR3141825}.

In dimension $d=2$, the exponents conjecturally satisfy the \emph{second KPZ relationship} $2\chi = \xi$. This second KPZ relationship is related to the fluctuations of the so-called Busemann functions that we define next. These functions were originally used by H.\,Busemann to study geodesics in metric geometry \citep{MR0075623}, and they were introduced in \fpp{} by Newman \cite[Theorem 1.1]{newman_surface_1995}. 
They have since found many applications in first- and last-passage percolation. For example, several groups have shown that bi-infinite geodesics cannot exist in certain directions under unproven differentiability hypotheses on the limit-shape that guarantee the existence of Busemann functions \cite{2016arXiv160902447A,newman_surface_1995,MR0075623}.

Let $\directions := \{ x \colon x_1 + x_2 = 1, x_1,x_2 > 0 \}$, the set of directions in $\R_{\geq 0}^2$ relevant to \lpp{}. Suppose $\{x_n \in \Z^2\}$ is such that $|x_n|_1 \to \infty$ but $x_n/|x_n|_1 \to x\in \directions$. A Busemann function in direction $x$ is defined by the limit 
\begin{equation}
    B^x(a,b) = \lim_{n \to\infty} G(a,x_n) - G(b,x_n) 
    \label{eq:busemann-fn-definition}
\end{equation}
if it exists. Let $\derivatives = \{ x \in \directions \colon g'(x) \text{ exists} \}$. For any $x \in \derivatives$, the limit in \eqref{eq:busemann-fn-definition} is expected to exist, and moreover, the resulting function is expected to satisfy four important properties: stationarity, integrability, corrector/recovery, and expectation duality (see \Defref{def:pre Busemann functions}). This result is known in \fpp{}\,\citep{2016arXiv160902447A}; however, for \lpp{}, it is only known conditional on unproven but mild differentiability hypotheses on the limit-shape\,\citep{MR3704768} (see Appendix \ref{sec:appendix diff of limit shape}). It is also expected that there is a unique Busemann function associated with each tangent line or gradient of the limit-shape.
We will not require uniqueness or \eqref{eq:busemann-fn-definition} in this paper, and thus prove our theorems for a larger class of pre-Busemann functions, which can be shown to exist without assuming unproven hypotheses on the time-constant or passage-time. 

A heuristic argument that was communicated to us by Newman, Alexander and others in a conference held at the American Institute for Mathematics in 2016, connects the covariance of Busemann increments to the second KPZ relationship. Consider a down/right lattice path along the antidiagonal defined as follows. Let $x_i = e_1$ for odd $i$, $x_i = -e_2$ for even $i$, $v_0 = 0$, and $v_k = \sum_{i=1}^k x_i$ be the $k$\textsuperscript{th} point on the down/right path. For some $x \in \derivatives$, suppose
\begin{align}
    \Cov(B^x(0, e_1),B^x(v_k, v_{k+1})) \leq 0 \quad \forall k \geq 1, \label{eq:negative correlation of busemann functions conjecture-0}\\
    \Cov(B^x(e_2, 0),B^x(v_k, v_{k+1})) \leq 0 \quad \forall k \geq 0. \label{eq:negative correlation of busemann functions conjecture}
\end{align}
Then, the argument shows that $2\chi \leq \xi$ (see Appendix \ref{sec:appendix busemann correlations and kpz relationship} for details).

In this paper, for \lpp{} with i.i.d.\ vertex weights having distribution $F$, mean $m$ and variance $\sigma^2$, the main theorem  (\Thmref{thm:main theorem criterion of negative covariance of adjacent busemann increments}) provides an easily verifiable sufficient condition that determines whether some \emph{adjacent Busemann increments} in \eqref{eq:negative correlation of busemann functions conjecture} with $k=0$ are negatively correlated for all $x \in \derivatives$. To demonstrate the use of this criterion, we show that i.i.d.\ Bernoulli weights with parameter $p > p^* \approx 0.6504$ have negatively correlated adjacent Busemann increments. The criterion is explicit: it only involves the large deviation rate function for i.i.d.\ sums of $F$. Theorem \ref{thm:main theorem criterion of negative covariance of adjacent busemann increments} is proved by establishing that the negative (resp.\ positive) covariance of adjacent Busemann increments for all $x \in \derivatives$ is equivalent to $g_F(x) \leq g_{\Exp}(x)$ ($g_F(x) \geq g_{\Exp}(x)$) for all $x \in \R_{\geq 0}^2$, where $\gexp$ is the function in \eqref{eq:exponential-lpp-limit-shape} with parameters $m$ and $\sigma$. Our criterion in \Thmref{thm:main theorem criterion of negative covariance of adjacent busemann increments} provides a sufficient condition for this last inequality involving time-constants to hold, and thus proves negative correlation.

\begin{question}[Newman, Alexander and others]
    Fix a down/right path and some Busemann function $B^u$. Are any two distinct Busemann increments of the form $B^u(y,y+e_1)$ or $B^u(z,z-e_2)$ on the down/right path negatively correlated? In particular, for the particular down/right path that goes along the main anti-diagonal, do \eqref{eq:negative correlation of busemann functions conjecture-0} and \eqref{eq:negative correlation of busemann functions conjecture} hold? 
\end{question}

\begin{question}
    Simulations (see Figure \ref{fig:time constants}) indicate that for some distributions $F$ ($\Uniform[0,1]$, $\Bernoulli(p)$), $g_F(x) \leq g_{\Exp}(x)$ for all $x \in  \R_{\geq 0}^2$; our result provides a sufficient condition that proves this for $\Bernoulli(p)$ for $p > p^* \approx 0.6504$. For some other distributions ($\operatorname{Lognormal}$ and $\chi^2(k)$, $k < 2$) it appears as though $g_F(x) \geq g_{\Exp}(x)$ for all $x \in  \R_{\geq 0}^2$; can this be proved for some class of distributions, in which case \Lemref{lem:negative covariance implied by time-constant dominance} shows that adjacent Busemann increments are positively correlated? Are there examples of distributions where neither $g_F(x) \leq \gexp(x)$ nor $\gexp(x) \leq g_F(x)$ hold for all $x \in \R_{\geq 0}^2$?
\end{question}

\section{Main Results}
\label{sec:main results}
Let \((\Omega,\mathcal{F},\Prob)\) be a probability space. Let $\Z^2$ act on $\W$ via a family of invertible, measure-preserving maps $\{T^z\}_{z \in \Z^2}$. The weights are given by a map $X \colon \W \to \R$, such that $\{X(T^z\w)\}_{z \in \Z^2}$ are i.i.d under $\Prob$. We use the shorthand $\w_z$ to refer to $X(T^z\w)$, the weight at $z \in \Z^2$. For a set \(I \subset \Z^2\) let \(I^{<} = \{x\in \Z^2 \mid z \not \prec x \ \forall z \in I\}\) be the set of lattice points which do not lie along a ray with initial point in \(I\) and angle in \([0,\pi/2]\). The set $\R^2_{\geq 0}$ refers to tuples $(x,y)$ such that both $x$ and $y$ are nonnegative.

Suppose \(w_0\) has moment generating function (mgf) \(M(t) = \E[ e^{\w_e t} ]\) that is finite for all $t \in (-\delta, \delta)$. Then, the corresponding large deviation rate function for i.i.d.\ sums of $F$ 
\begin{equation}
    I(a) = \sup_{t} \{ at - \log M(t) \},
    \label{eq:large deviation rate function definition}
\end{equation}
is finite for all $a$ in some nontrivial interval. Thus, in this paper, we assume in addition to the nondegeneracy of $F$, that 
\begin{equation}
    \label{eq:mgf and nontriviality assumption}
    \exists \delta > 0 ~\suchthat. ~ M(t) < \infty \quad \forall t \in  (-\delta, \delta) .
\end{equation}
Next, we define the central objects of study in this paper.
\begin{define}[pre-Busemann functions]
    \label{def:pre Busemann functions}
    Let 
    \[ 
        \grads = \{ \grad g (x) \colon g(x) \text{ is differentiable at } x\in \mathcal U^{\circ} \},
    \]
    and let $\mathcal{K}$ be some index set such that there is a function $\phi \colon \indexset \to \grads$ that is onto. A set of functions $\{B^u\}_{u \in \mathcal{K}}$, $B^u \colon \W \times \Z^2  \times \Z^2 \to \R$ is a called a family of pre-Busemann functions if each $B^u$ satisfies the following properties:
   \begin{enumerate}
        \item Stationarity: for \(\mathbb{P}\)-a.e.~\(\omega\) and all \(x,y,z\in \Z^2\), \(B^u(\omega, z+x,z+y) = B^u(T^z \omega ,x,y)\).
        \item Additivity: for \(\mathbb{P}\)-a.e.~\(\omega\) and all \(x,y,z\in \Z^2\), 
            \begin{equation}
                B^u(\omega, x,y)+ B^u(\omega,y,z)=B^u(\omega,x,z).
                \label{eq:busemann function additivity}
            \end{equation}
        \item Past Independence: For any \(I\subset \Z^2\), the variables 
        \[
            \{(\omega_x,B^u(\omega ,x,y)) \colon x\in I \}
        \]
        are independent of \(\{\omega_x \colon x\in I^{<} \}\).
        \item Corrector/Recovery property: for \(\mathbb{P}\)-a.e.~\(\omega\) and all \(x\in \Z^2\)
        \begin{align}
            \w_x = \min_{i=1,2} \left( B^{u}(\omega ,x,x+e_i)  \right).
            \label{eq:busemann function dynamic programming or recovery property}
        \end{align}
        \item Expectation duality: 
        \begin{equation}
            \E[B^u(\w,0,e_1),B^u(\w,0,e_2)] = \phi(u).
            \label{eq:correct expectation of busemann functions}
        \end{equation}
   \end{enumerate}
\end{define}

\begin{remark}
    A typical choice for $\mathcal{K}$ and $\phi$ in \Defref{def:pre Busemann functions} is $K = \derivatives$ and $\phi(u) = \grad g(u)$ since under the conditions described below, pre-Busemann functions are Busemann functions satisfying \eqref{eq:busemann-fn-definition} for each $u \in \derivatives$. This is also the reason we call \eqref{eq:correct expectation of busemann functions}, ``Expectatation duality''. We will drop the notation $\w$ from $B(\w,x,y)$ if it is not relevant to the argument.
\end{remark}

Theorem 5.2 in \cite{MR3704768} originally established the existence of a family of pre-Busemann functions under the assumption that $\Prob(\omega_0\geq c ) = 1$ for some $c \in \R$, and $\E[|\w_0|^p] < \infty$ for some $p > 2$. Theorem 4.8 in \citep{MR4089495} removed the requirement that the weights were bounded below a.s. Thus, the remaining finite-moment assumption is implied by the finite mgf assumption \eqref{eq:mgf and nontriviality assumption} in this paper. We state the relevant parts of \citep[Theorem 5.2]{MR3704768} and \citep[Theorem 4.8]{MR4089495} next, leaving some details unspecified.

Let \(\Omega = \R^{\Z^{2}}\), \(\mathcal F\) be the Borel \(\sigma\)-algebra, and $\Prob$ be the i.i.d.\ measure with distribution $F$. 
Define the shift maps \(\{T^x\}_{x\in \Z^2}\) which act on \(\Omega\) by \((T^x \omega)_y=\omega_{x+y}\) for \(x,y\in \Z^2\). The family of pre-Busemann functions \(\{B_{\pm}^{u}(\hat \omega ,x,y)\}_{u \in \directions}\) are defined on an explicit extended space \(\hat \Omega = \Omega \times \Omega' = \Omega \times \R^{\{1,2\}\times \mathcal A_0 \times \Z^2}\), where \(\mathcal A_0\) is a countable dense subset of the interval \((m,\infty)\). Let $\pi \colon \hat\Omega \to \W$ represent projection onto the $\w$ coordinate. We define a family of measurable, commuting, invertible translation maps $\{\hat T^z\}_{z \in \Z^2}$, acting on $\hat\W$ as follows: Let $\hat\w \in \hat \W$ be written as $(\w,\w')$ where $\w = (\w_x)_{x \in \Z^2}$ and $\w' = (\w_x^{\alpha,i})_{\alpha \in \mathcal{A}_0,i \in \{1,2\}, x \in \Z^2}$. Let $(\hat T^x \hat\w)_y = \hat\w_{x+y}$ for all $x,y \in \Z^2$, where $\hat \w_x = (\w_x,\w'_x)$. It is clear that $\pi \circ \hat T^z = T^z \circ \pi$ for all $z \in \Z^2$; i.e., $\hat T$ intertwines correctly with $T$ under projection.


\begin{theorem*}[\citep{MR3704768}, Theorem 5.2 and \citep{MR4089495}, Theorem 4.8]
    There exist real-valued Borel functions \(B_+^{u}(\hat \omega ,x,y)\) and \(B_-^{u}(\hat \omega ,x,y)\) 
    of \( (\hat \omega, u,x,y) \in \hat \Omega \times \directions \times \Z^2 \times \Z^2\) and a 
    translation invariant Borel probability measure \(\hat{\mathbb{P}}\) such that the following properties hold.
    \begin{enumerate}
        \item Under \(\hat{\mathbb{P}} \), the marginal distribution of the configuration \(\omega = \pi(\hat \w)\) is the i.i.d.\ measure \(\mathbb{P}\). For each \(u \in \directions\) and \(\pm\), the \(\R^3\)-valued process \(\{\psi_x^{\pm,u}\}_{x\in \Z^2}\) defined by 
        \begin{align*}
            \psi_x^{\pm,u}(\hat{\omega})= (\omega_x,B_{\pm}^{u}(\hat \omega ,x,x+e_1),B_{\pm}^{u}(\hat \omega ,x,x+e_2) )
        \end{align*}
        is stationary under translation by \(\hat T^{e_1}\) or \(\hat T^{e_2}\). 
        \item For any \(I\subset \Z^2\), the variables 
        \[
            \{(\omega_x,B_+^{u}(\hat \omega ,x,y),B_-^{u}(\hat \omega ,x,y )) \mid x\in I,\ u \in \directions, \ i \in \{1,2\} \}
        \]
        are independent of \(\{\omega_x \mid x\in I^{<}\}\).
        \item \(\{B_{s}^{u}\}_{(u,s) \in \directions \times \{\pm\}}\) is a family of pre-Busemann functions satisfying properties $(1)-(5)$ in \Defref{def:pre Busemann functions}. Here, the index set is $\indexset = \directions \times \{\pm\}$, and $\phi$ maps $(u,s)$ to a unique, explicit member of the subgradient of $g_F$ at $u$.
    \end{enumerate}
\end{theorem*}
Under mild differentiability hypotheses on \(g_F(x)\), the pre-Busemann functions $B^u_{\pm}$ are measurable with respect to the completion of \(\mathcal F\), and the limit in Equation \eqref{eq:busemann-fn-definition} exists (see Theorem 5.3 in \cite{MR3704768} and Appendix \ref{sec:appendix diff of limit shape}). The theorems we prove in this paper are valid for all families of pre-Busemann functions, of which families of Busemann functions satisfying \eqref{eq:busemann-fn-definition} are a subset.

\begin{theorem}
    Let \(F\) have mean \(m\), variance \(\sigma^2 > 0\), and mgf satisfying \eqref{eq:mgf and nontriviality assumption}. Let $\{B^u\}_{u \in \indexset}$ be a family of pre-Busemann functions satisfying the conditions in \Defref{def:pre Busemann functions}. Suppose
    \begin{equation}
        \log\left(4\right)  \frac{s}{s + 1} < I \left( \frac{\gexp(1,s)}{1+s} \right) \quad \forall s \in (0,1),
    \end{equation}
    where $\gexp(1,s)$ is the function in \eqref{eq:exponential-lpp-limit-shape} with parameters $m$ and $\sigma$, and $I$ is the rate function for i.i.d.\ sums of $F$. Then,
    \begin{equation*}
    \Cov ( B^u(e_2,0), B^u(0,e_1) ) \leq 0 \quad \forall u \in \indexset,
    \end{equation*}
    \label{thm:main theorem criterion of negative covariance of adjacent busemann increments}
\end{theorem}
Theorem \ref{thm:main theorem criterion of negative covariance of adjacent busemann increments} follows from Lemma \ref{lem:negative covariance implied by time-constant dominance} and Theorem \ref{thm:rate function criterion for time-constant dominance} below.
\begin{lemma}
    Under the conditions of \Thmref{thm:main theorem criterion of negative covariance of adjacent busemann increments}, let $g_F$ be the time-constant.
    Then,
    \begin{equation}
        \begin{aligned}
           & g_F(x) \leq \gexp(x) \quad \forall x \in \R_{\geq 0}^2 \\
           & \qquad \qquad \Longleftrightarrow \\
           & \Cov(B^u(e_2,0), B^u(0,e_1) ) \leq 0 \quad \forall u \in \indexset. 
       \end{aligned}
           \label{eq:negative-covariance implied by time-constant dominance}
   \end{equation}
    The equivalence in \eqref{eq:negative-covariance implied by time-constant dominance} holds with both inequalities reversed, thus providing an analogous equivalence for positive correlation.
   \label{lem:negative covariance implied by time-constant dominance}
\end{lemma}

The second ingredient needed for proving \Thmref{thm:main theorem criterion of negative covariance of adjacent busemann increments} is the following theorem, which gives a sufficient condition on the rate function for i.i.d.\ sums of $F$ that determines when $g_F(x) \leq K$ for some $K > m$. 
\begin{theorem}
Under the conditions of \Thmref{thm:main theorem criterion of negative covariance of adjacent busemann increments}, let $s \in (0,\infty)$ and $K > m$ be such that
\begin{equation}\label{eq: rate function criterion for time-constant dominance}
    \log \left(4\right) \frac{s}{s + 1} < I\left(\frac{K}{1+s} \right).
\end{equation}
Then, we have $g_F(1,s) \leq K$.
\label{thm:rate function criterion for time-constant dominance}
\end{theorem}

\begin{remark}
The statement of \Lemref{lem:negative covariance implied by time-constant dominance} suggests the use of a convex ordering inequality in a manner similar to \citet{van_den_berg_inequalities_1993}. Given two distributions $F_1$ and $F_2$, we say $F_2$ is \textbf{more variable} than $F_1$ and write $F_1 \ll F_2$ if
\begin{equation}
    \int \phi dF_1 \leq \int \phi dF_2
\end{equation}
for all convex, non-decreasing integrable functions $\phi$. If \(F_1 \ll F_2\), since the \lpt{} is a convex non-decreasing function of the edge-weights, the associated time-constants satisfy $g_1(x) \leq g_2(x) ~\forall x \in \R_{\geq 0}^2$. For this to apply in our case, we would need to find a distribution $F$ with the same mean and variance as an exponential distribution $\Exp(\lambda)$, such that $F \ll \Exp(\lambda)$. Unfortunately, the following proposition shows this to be impossible.
\end{remark}

\begin{prop}
    Let \(X\) and \(Y\) be random variables with different distributions \(F\) and \(G\). If \(E[X]=E[Y]\) and \(E[X^2] = E[Y^2] \) then \(G \ll F\) cannot hold. 
    \label{prop:failure of convex ordering when matching moments}
\end{prop}

Next, we utilize the criterion in \Thmref{thm:main theorem criterion of negative covariance of adjacent busemann increments} to demonstrate a distribution for which adjacent Busemann increments are negatively correlated.

\begin{prop}
Consider \lpp{} with i.i.d.\ $\Bernoulli(p)$ weights. If $p > p^* \approx 0.6504$ (see \Propref{prop: bernoulli inequality} for the definition of $p^*$), then the criterion in \Thmref{thm:main theorem criterion of negative covariance of adjacent busemann increments} is satisfied.
\label{prop:bernoulli weights}
\end{prop}
Numerical computations show that the condition in \Thmref{thm:main theorem criterion of negative covariance of adjacent busemann increments} holds for $\Bernoulli(p)$ weights for all parameters $p \geq 1/2$, and we believe that \Propref{prop:bernoulli weights} can be extended to this setting (see \Figref{fig:coarse-grid-ber-25-r0}). It follows trivially that shifted and scaled $\Bernoulli$ random variables also have negatively correlated Busemann function increments. 
\begin{cor}
    Any random variable \(Y\) with \(P(Y=a)=1-p\) and \(P(Y=b) =p\) for \(p^*<p<1\) as in \Propref{prop: bernoulli inequality}, also satisfies the criterion in \Thmref{thm:main theorem criterion of negative covariance of adjacent busemann increments}.
\label{cor:criterion for general shifted bernoulli}
\end{cor}

\begin{figure}[!htpb]
    \centering
    \begin{subfigure}[b]{0.32\textwidth}
            \includegraphics[width=\textwidth]{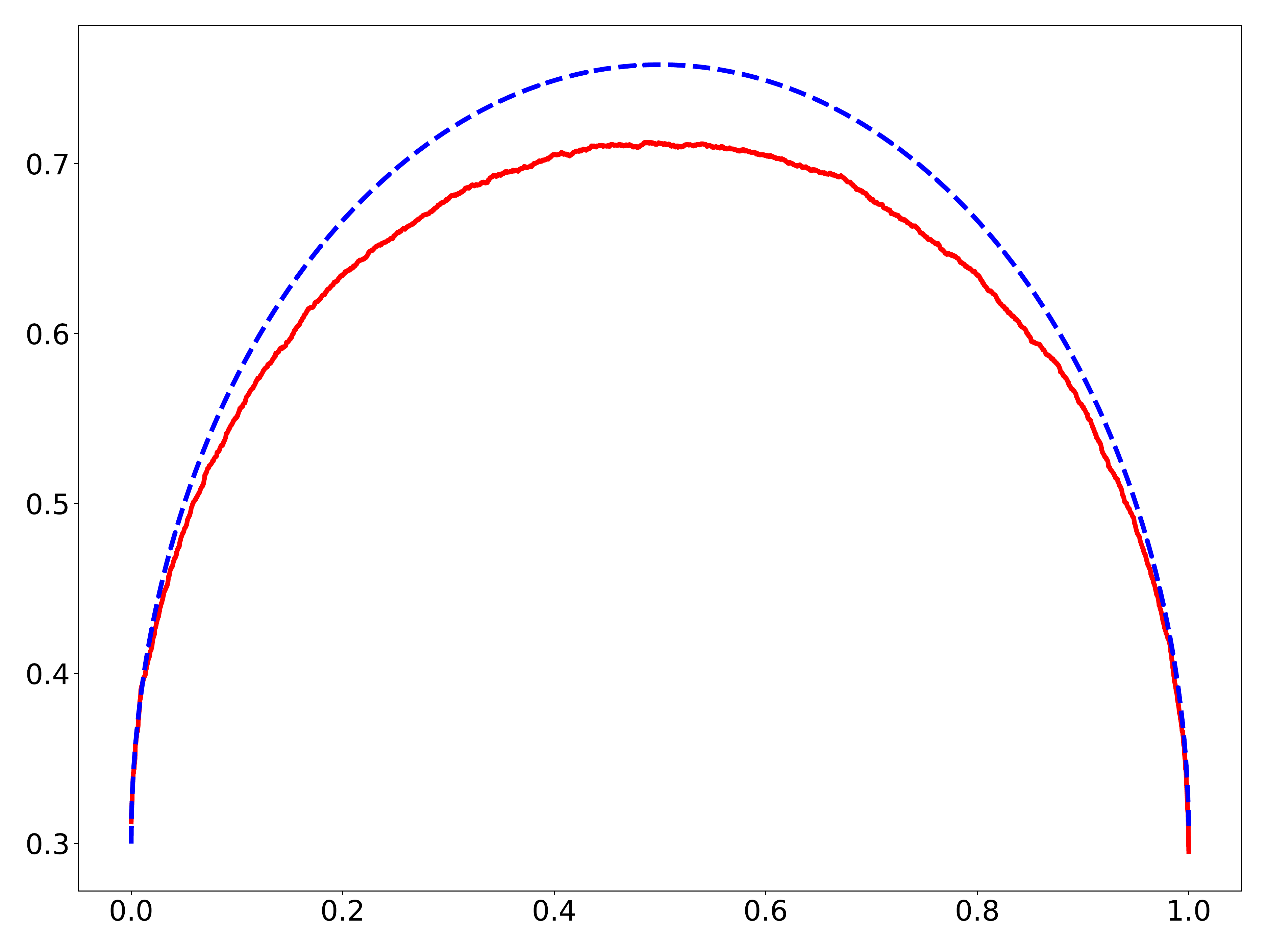}
            \caption{$F =\Bernoulli(.3)$}
    \end{subfigure}
    \begin{subfigure}[b]{0.32\textwidth}
            \includegraphics[width=\textwidth]{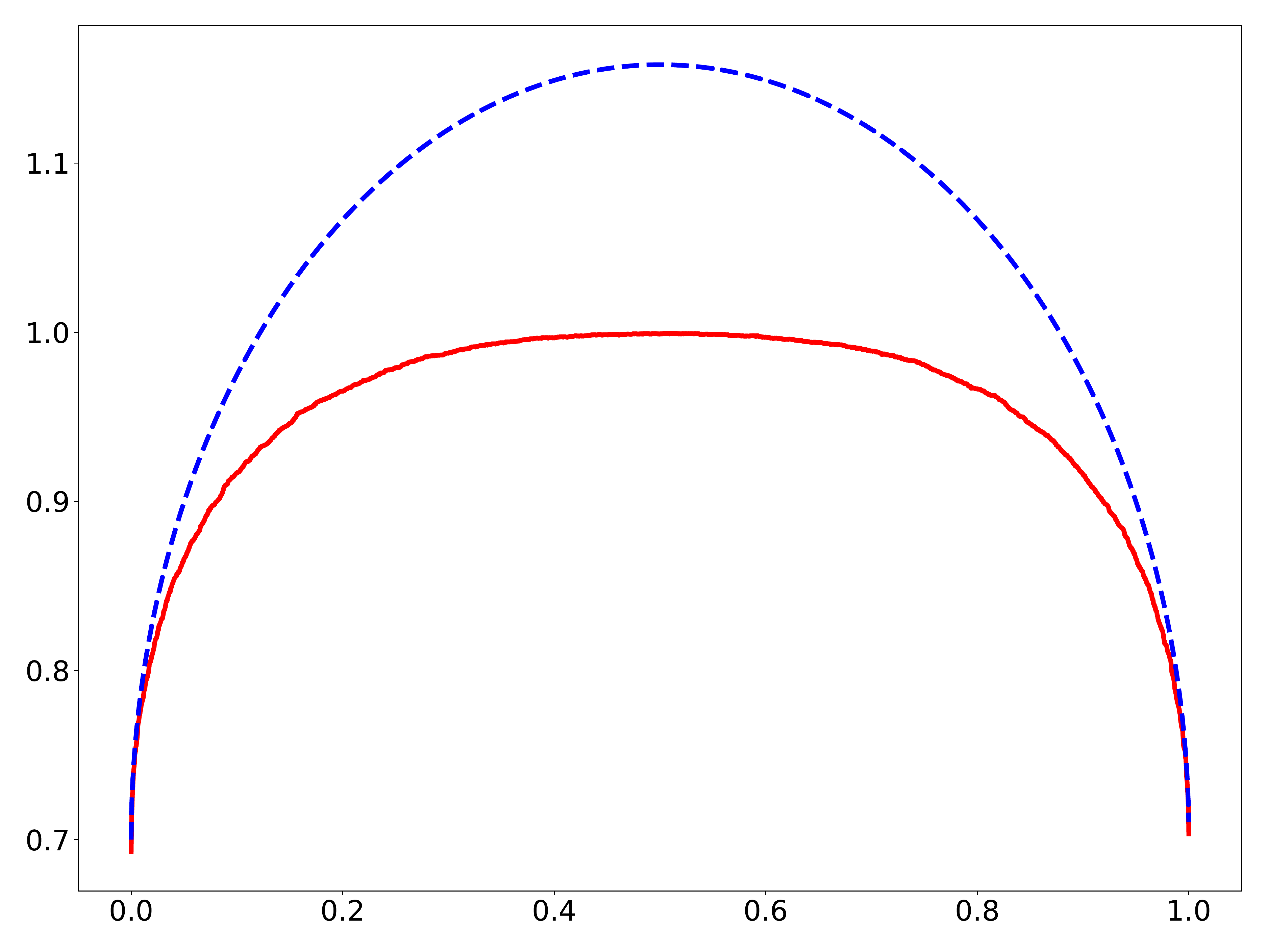}
            \caption{$F =\Bernoulli(.7)$}
    \end{subfigure}
    \begin{subfigure}[b]{0.32\textwidth}
        \includegraphics[width=\textwidth]{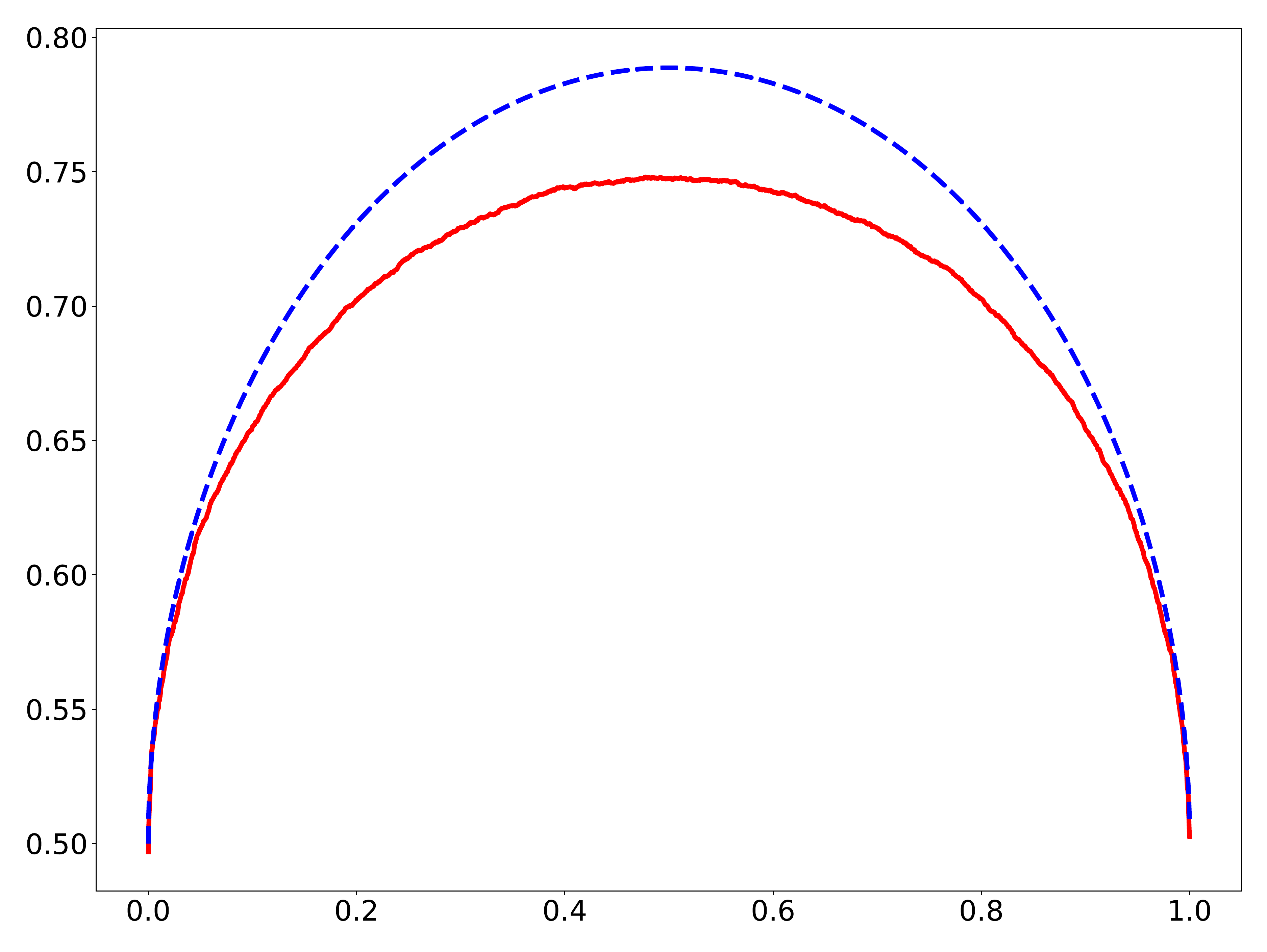}
    \caption{$F =\text{Unif}[0,1]$} 
    \end{subfigure}
    \newline
    \begin{subfigure}[b]{0.32\textwidth}
        \includegraphics[width=\textwidth]{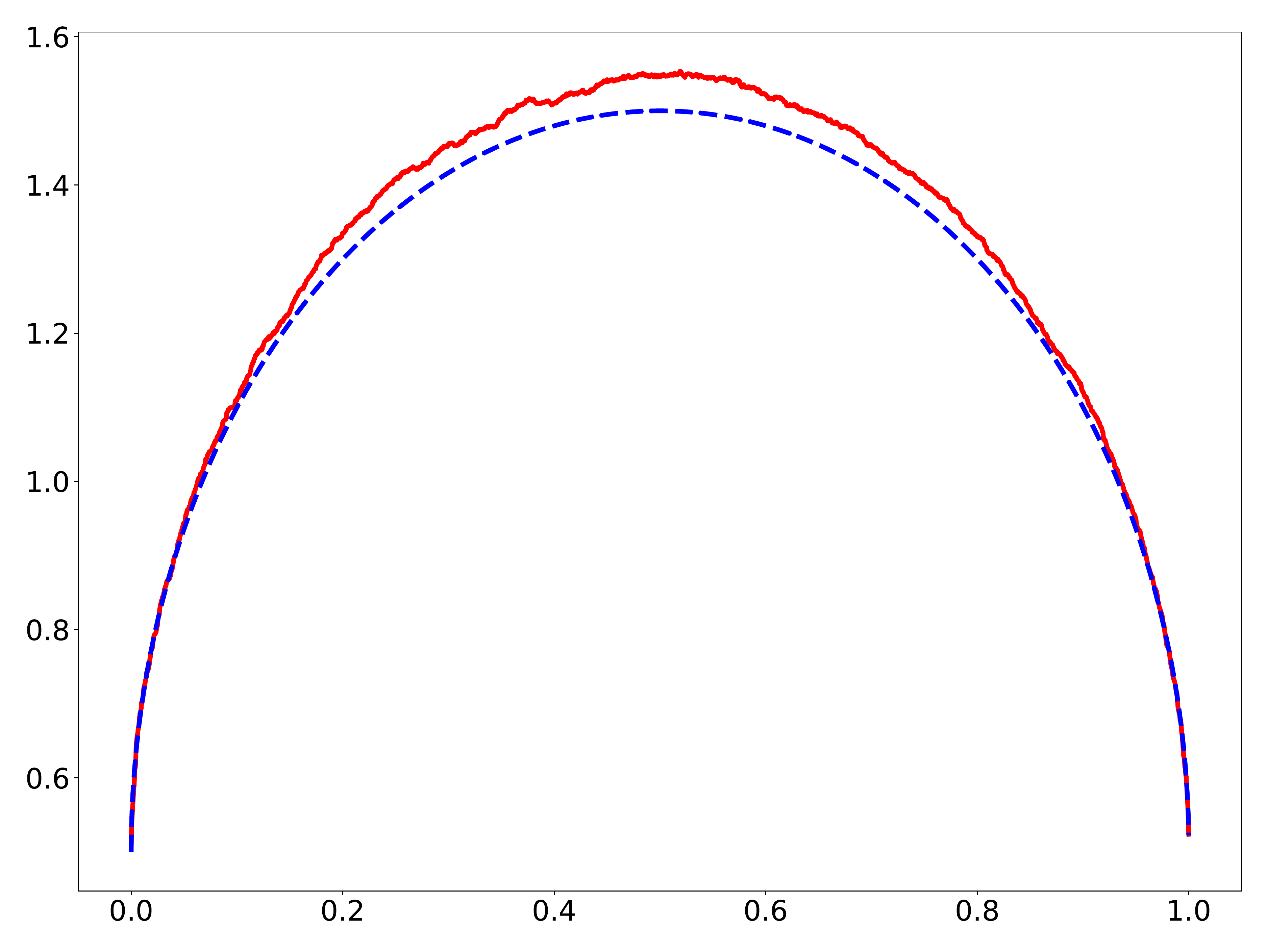}
        \caption{$F =\chi^2(.5)$}
\end{subfigure}
\begin{subfigure}[b]{0.32\textwidth}
        \includegraphics[width=\textwidth]{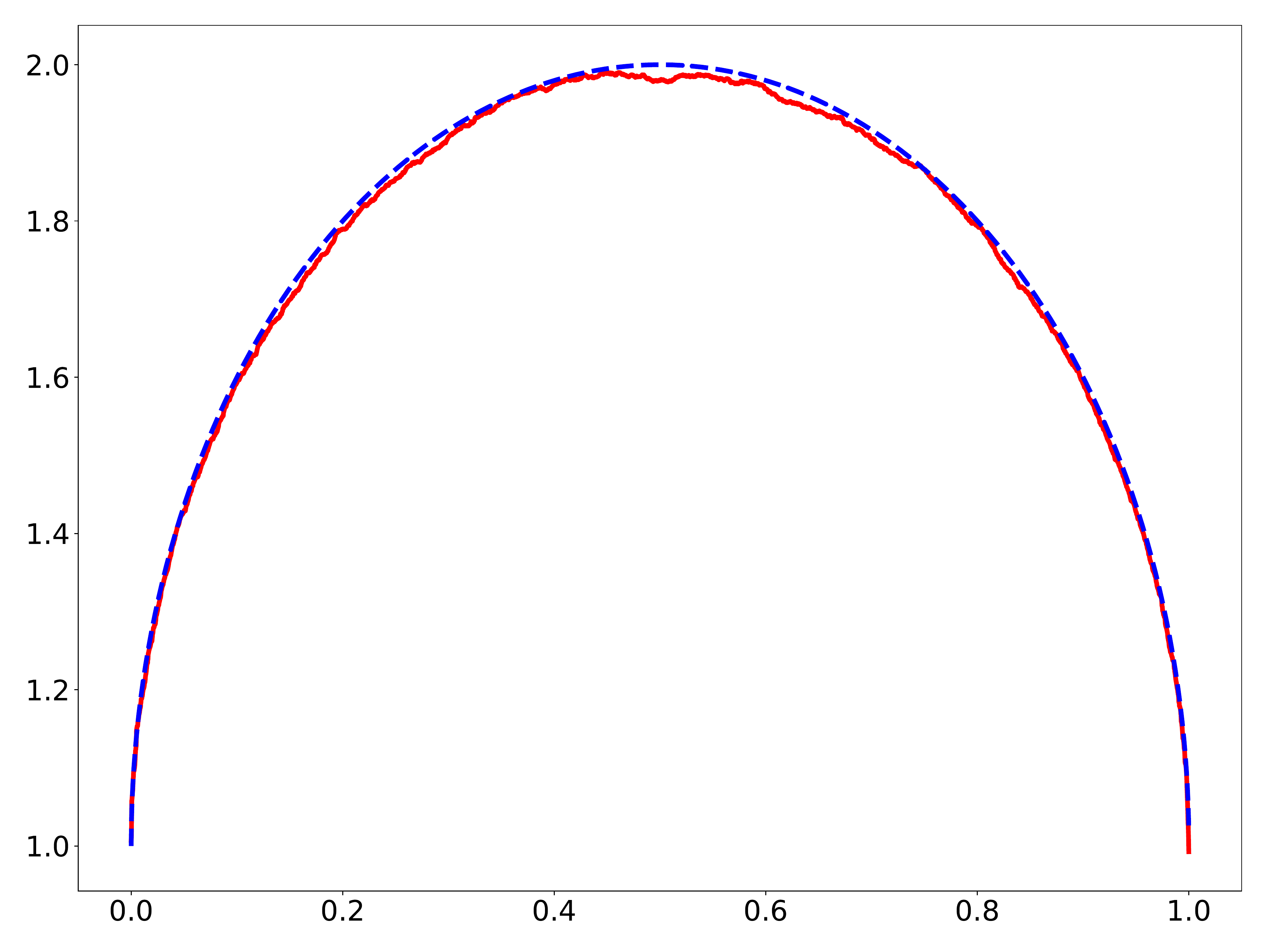}
        \caption{$F =\Exp(1)$}
\end{subfigure}
\begin{subfigure}[b]{0.32\textwidth}
        \includegraphics[width=\textwidth]{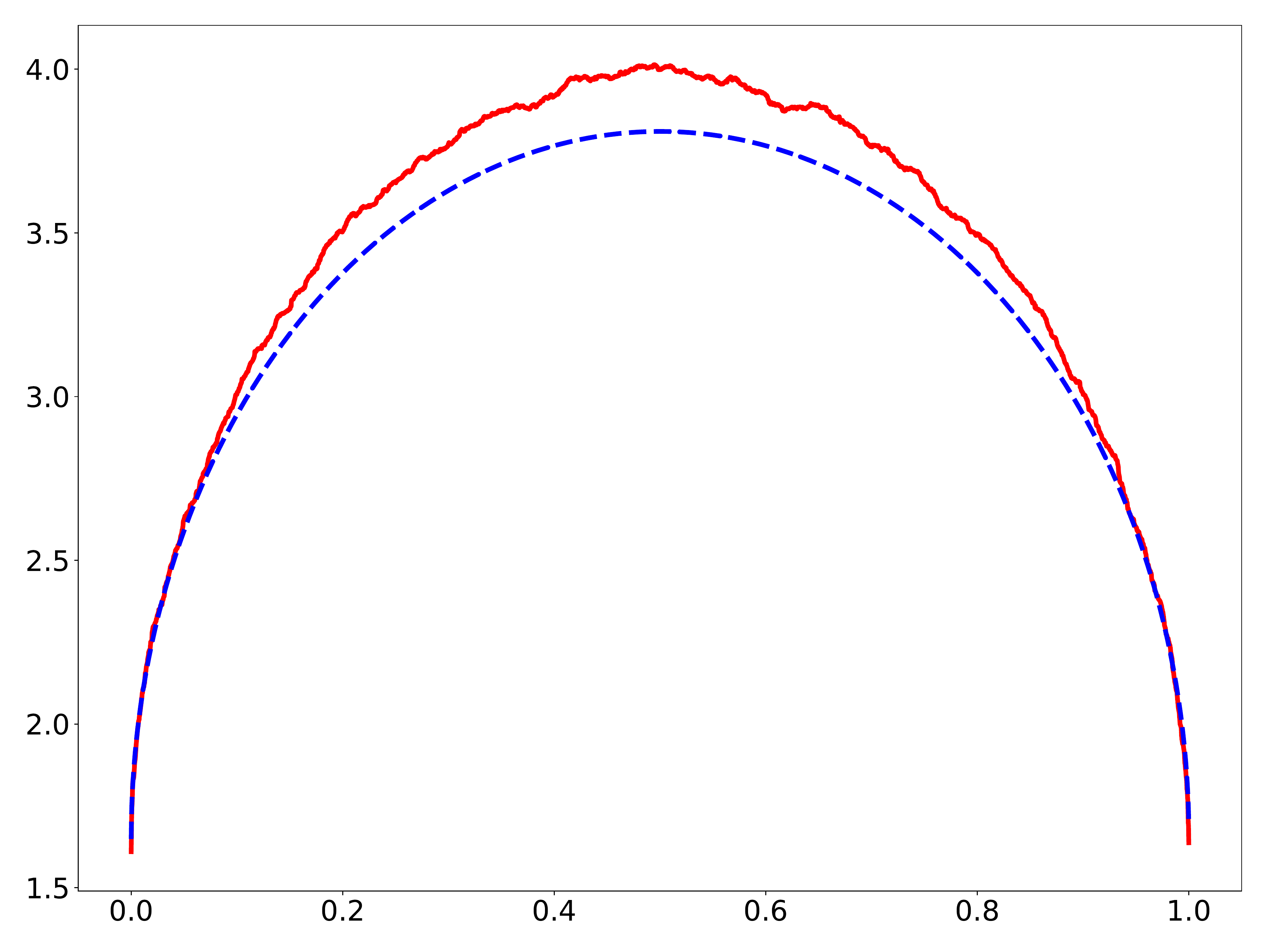}
        \caption{$F =\text{Lognormal}[0,1]$}
\end{subfigure}
    
    \caption{The red curves in the figures show simulated time-constants $g_F(x,1-x)$ for $x \in [0,1]$ for various distributions \(F\). Scaled passage times $G([Nx],[N(1-x)])/N$ for $N = 8000$ were used to approximate $g_F(x,1-x)$. The dashed blue curve shows the time-constant for exponential weights. Note that for the $\Uniform$ and $\Bernoulli$ distributions, we appear to have $g_F \leq \gexp$, and for $\Lognormal$ and $\chi^2(0.5)$ weights, we appear to have $g_F \geq \gexp$.}
    \label{fig:time constants}
\end{figure}

\oldarjun{Could we show that Uniform distribution has this property? We tried comparison inequalities and convex ordering, but nothing appeared to work. Specifically, we tried one could utilize the Berg-Kesten convex dominance method to compare $F_1$ with a $\Bernoulli(p)$ distribution $F_2$ that has support in $[a',b']$ that has the same mean and variance as $F_1$, but $F_1 \ll F_2$}

\begin{remark}
Since our coarse graining method is far from optimal, the criterion fails for exponentially distributed weights. In the exponential case, the covariance of any two distinct Busemann increments on a down/right path is $0$; in fact, they are independent.
\end{remark} 
\section{Proofs}

\subsection{Covariance and the time-constant: Proof of Lemma \ref{lem:negative covariance implied by time-constant dominance}}

Consider a slice of the time-constant
\[\gamma(s) = \begin{cases} g(1,s) \qquad & \text{if}\qquad 0 \leq s < \infty, \\-\infty \qquad & \text{if} \qquad s<0 \end{cases}. \]
In \lpp{} the time-constant is concave and \(-\gamma(s)\) is a convex function. 
Consider the Legendre transform (with a change of coordinates) of \(-\gamma(s)\) given by
\begin{align*}
    f(a) =  \sup_{s>0} \left( -sa  + \gamma(s)\right).
\end{align*}
From the trivial bound $\gamma(s) > m(1+s)$ for $s > 0$, it follows that $f(a)$ is only finite for \(a > m\). 
From Legendre duality, we have
\begin{align}\label{eq:legendre dual for time-constant gamma s}
    \gamma(s) & = \inf_{a > m} \left( sa + f(a) \right).
\end{align}
Since $g(\lambda x) = \lambda g(x)$, we must have $\partial_{\lambda} g(\lambda x) = \grad g(\lambda x) \cdot x$, and hence $g(x) = \grad g(\lambda x) \cdot x$ for all $\lambda > 0$. In terms of $\gamma(s)$, this translates to $\gamma(s) = \grad g(1,s) \cdot e_1 + s\gamma'(s)$.  From Legendre duality, it follows that when $\gamma(s)$ is differentiable at $s$, we get $f(\gamma'(s)) = \gamma(s) - s \gamma'(s) = \grad g(1,s) \cdot e_1$. Therefore,
\oldarjun{Do we Need a reference for the statement that says that when $\gamma(s)$ is differentiable at $s$, we get $f(\gamma'(s)) = \gamma(s) - s \gamma'(s) = \grad g(1,s) \cdot e_1$}
\[
    \grad g(1,s) = (f(a),a),
\]
where $a = \gamma'(s)$. The set of derivatives of \(\gamma(s)\) are in one-to-one correspondence with gradients of the time-constant. Thus, when we write $a \in \grads$ below, we mean $(f(a),a) \in \grads$, where $\grads$ is the set of gradients of the time-constant.

\begin{lemma}\label{lem: adjacent increments correlation}
    Let \(F\) have mean \(m\) and variance \(\sigma^2\).
    The covariance of all adjacent increments in a pre-Busemann family is \emph{negative}; i.e., $\Cov(B^{u}(0,e_1),B^{u}(e_2,0)) \leq 0$ for all $u \in \indexset$, if and only if
\begin{equation}
    f(a) \leq m + \frac{\sigma^2}{a - m} \qquad \forall a \in \grads.
    \label{eq:condition on legendre dual for positive covariance}
\end{equation}
\end{lemma}

\begin{proof}
    Fix $a \in \grads$. From \Defref{def:pre Busemann functions}, let $u \in \indexset$ be such that the pre-Busemann function $B^u$ satisfies
    \[
        \E[B^u(0,e_1),B^u(0,e_2)] = (f(a),a).
    \]
    From the additivity property \eqref{eq:busemann function additivity}, we have
    \[
    B^u(x,x+e_2) + B^u(x+e_2,x+e_1) = B^u(x,x+e_1).
    \]
    Inserting this into the recovery property \eqref{eq:busemann function dynamic programming or recovery property}, setting $x = 0$, and rearranging, we get
\begin{align}
    B^{u}(0,e_1)  & = \w_0 + B^{u}(e_2,e_1)^+  \label{eq:dynamic programming busemann increment 1},\\
    B^{u}(0,e_2) & = \w_0 + B^u(e_2,e_1)^- \label{eq:dynamic programming busemann increment 2}.
\end{align}
where $f^\pm = \max(\pm f,0)$. It follows from \eqref{eq:dynamic programming busemann increment 1} and \eqref{eq:dynamic programming busemann increment 2} that
\[
    \E[(B^{u}(e_2,e_1)^+,B^{u}(e_2,e_1)^-] = (f(a) -m,a - m).
\]
Thus, the covariance of $B^{u}(0,e_1) \AND B^{u}(e_2,0)$ can be written as
\begin{equation}
\begin{aligned}
    \Cov(B^{u}(0,e_1),B^{u}(e_2,0)) 
    & = -\Cov(\w_0,\w_0) - \Cov(B^{u}(e_2,e_1)^+,B^{u}(e_2,e_1)^-)\\
    & = - \sigma^2 + \E[B^{u}(e_2,e_1)^+]\E[B^{u}(e_2,e_1)^-]\\
    & = - \sigma^2 + (f(a) - m) (a - m),
\end{aligned}
\label{eq:covariance of adjacent busemann increments}
\end{equation}
using bilinearity of covariance and the fact that $B^{u}(e_2,e_1)$ is independent of the weight $\w_0$ (see \citep[Theorem 3.3]{MR3704769} and (3) in \Defref{def:pre Busemann functions}). Equations \eqref{eq:dynamic programming busemann increment 1},\eqref{eq:dynamic programming busemann increment 2} and \eqref{eq:covariance of adjacent busemann increments} are due to T.\,Seppalainen \citep{seppalainen2018}, who noted that zero-correlation of all adjacent Busemann increments (eq.\ \eqref{eq:covariance of adjacent busemann increments} is identically $0$) implies that the limit-shape must be given by \eqref{eq:exponential-lpp-limit-shape}. Equation \eqref{eq:covariance of adjacent busemann increments} shows that the covariance is negative if and only if 
\begin{align*}
    f(a) & \leq m + \frac{\sigma^2}{a-m}.
\end{align*}
\end{proof}

\begin{remark}
    The proof of \Lemref{lem: adjacent increments correlation} also shows that $\Cov(B^{u}(0,e_1),B^{u}(e_2,0)) \geq 0$ for all $u \in \indexset$ if and only if
\begin{equation}
    f(a) \geq m + \frac{\sigma^2}{a - m}.
\end{equation}
    Thus, the following proof of \Lemref{lem:negative covariance implied by time-constant dominance} shows that covariance $0$ for all adjacent increments is equivalent to the fact that $\gexp(x)$ is the time-constant.
\end{remark}

\begin{proof}[Proof of Lemma \ref{lem:negative covariance implied by time-constant dominance}]
First suppose that all pre-Busemann increments are negatively correlated. Combining \eqref{eq:legendre dual for time-constant gamma s} and Lemma \ref{lem: adjacent increments correlation}, we find
\begin{equation}
    \begin{aligned}
    g_F(1,s) = \gamma(s) & = \inf_{m < a < \infty} \left( f(a) + s a \right) \\
    & \leq \inf_{m < a < \infty} \left(  m + \frac{\sigma^2}{a - m} + s a \right)\\
    & = m(1+s)+ 2\sigma \sqrt{s} = \gexp(1,s).
    \end{aligned}
    \label{eq:inequality of time-constant wrt universal exponential limit shape} 
\end{equation}
By the 1-homogeneity, continuity and symmetry of $g_F$ and $\gexp$, it follows that $g_F(x) \leq \gexp(x) ~\forall x \in \R_{\geq 0}^2$ (see \Propref{prop:extending time constant inequality from 0-1 interval to entire first quadrant} for details).

Next, suppose \(g_F(1,s)\leq \gexp(1,s)\) for all $s \in [0,\infty)$. Then, for $m < a < \infty$, we have
\begin{align*}
    f(a)& = \sup_{s>0} \left(\gamma(s) - sa \right) \\
    & \leq \sup_{s>0} \left(  m(1+s)+ 2\sigma \sqrt{s} - sa \right) \\
    & = m + \frac{ \sigma^2}{a-m}
\end{align*}
\end{proof}

\subsection{Coarse graining argument: Proof of Theorem \ref{thm:rate function criterion for time-constant dominance}}
Recall that $I$ is the large deviation rate function for i.i.d.\ sums of $F$. In this section, we prove Theorem \ref{thm:rate function criterion for time-constant dominance}, which states that if for any $s \in (0,\infty)$ and $K > m$, we have
\begin{align*}
    \frac{\log(4)s}{(1+s)} < I\left(\frac{K}{1+s} \right),
\end{align*}
then \(g_F(s) \leq K\). 

Let $s \in (0,\infty)$. Since $\log(4)s/(1+s) < I(K/(1+s))$, choose $r \in \Q_{>0}$ such that 
\begin{equation}
    \frac{\log (4) (s+r)}{(1+s)(1+r)} 
    < I\left(\frac{K}{1+s} \right).
    \label{eq:choose r small enough}
\end{equation}
We define an event on which the weights in $[0,N] \times [0,Ns]$ are not too large:
\begin{align}
   \good_N =\bigcup_{\substack{1\leq i \leq N \\ 1 \leq j \leq Ns }}\{\omega: -b_N \leq \omega_{ij}\leq b_N\}. \label{eq: definition good}
\end{align}
Since \(F\) has mgf \(M(t)\) which is finite for \(t\in (-\delta, \delta)\) where \(\delta>0\), applying a union bound to \eqref{eq: definition good} gives
\begin{align}
    P\left(\good_N^C \right) \leq cN^2 se ^{-\lambda b_N}
    \label{eq:bound on good complement using mgf}
\end{align}
for constants \(\lambda>0\) and \(c>0\) that only depend on $F$.

Let \(\path_N\) be the set of all up/right paths from \((0,0)\) to \(N(1,s)\). Consider the event
\begin{align}\label{eq:naive set}
    A_N = \bigcup_{\Gamma \in \path_N}\{ \omega \in \Omega:  G(\Gamma)> N  K\}.
\end{align}
If \(\lim_{N\to\infty} P (A_N)  =0\), then we have
\begin{align}
g_F(s) \leq K.
\end{align}
To show that $A_N$ has vanishing probability as $N \to \infty$, we use a coarse graining argument to reduce the number of allowed paths in \(A_N\) (entropy reduction), and then use a union bound. This strategy is inspired by \citet{MR3580031}. 

Let \(M\) be a positive integer (to be fixed after the proof of \Lemref{lemma: good path}) such that $Mr$ is an integer as well. For \(k\in \Z_{\geq 0}\) define the anti-diagonal lines \(\antidiaglines_k=\{ (x,y)\in \R^2:  y=r (k M-x)\}\), and \(\antidiaglines = \cup_{k\in \Z_{\geq 0}} \antidiaglines_k\).
 Let \(L \leq M\) be any positive real number such that $Lr$ is an integer, and let \(C_g\) be the \textbf{coarse grid} consisting of points in \(\antidiaglines \cap \Z^2\) a diagonal distance \(L\sqrt{1+r^2}\) apart; i.e., 
\[
    C_g=\{ p \in \antidiaglines \cap \Z^2:  |p-(kM,0)|_1 = q L (1 + r) \text{ for } k,q \in \Z_{\geq 0} \}. 
\]
Here we use \(|x|_p\) to denote the \(\ell^p\) norm.

Define the \textbf{free zone} \(F_g\) with 
\[
    F_g = \bigcup_{k \geq 1} F_g^k := \bigcup_{k \geq 1} \{(x,y)\in \Z^2 : r(kM-x) \leq y \leq r(kM-x)+Lr \},
\]
and define two extra sets of anti-diagonal lines that flank each $\antidiaglines_k$:
$$
    \antidiaglines_k^\pm := \{ (x,y) \in \R_{\geq 0}^2 \colon y = r (kM - x) \pm L r \} \quad k \in \Z_{\geq 0}
$$
$F_g^k$ is the set of lattice points between $\antidiaglines_k$ and $\antidiaglines_k^+$.

We form another set of paths $\pathp_N$ by considering up/right paths from $(0,0)$ to $(N,Ns)$ that 
\begin{enumerate}
    \item only intersect lines in $\antidiaglines$ at coarse grid points $C_g$ except for the final point $(N,Ns)$, and
    \item are up/right everywhere except for the free zones, where they are allowed to move in all 4 directions and have $\ell^1$ length at most \(2L\) in each \(F_g^k\).
\end{enumerate}
\begin{figure}[!htbp]
    \centering
    \includegraphics[width=0.5 \linewidth]{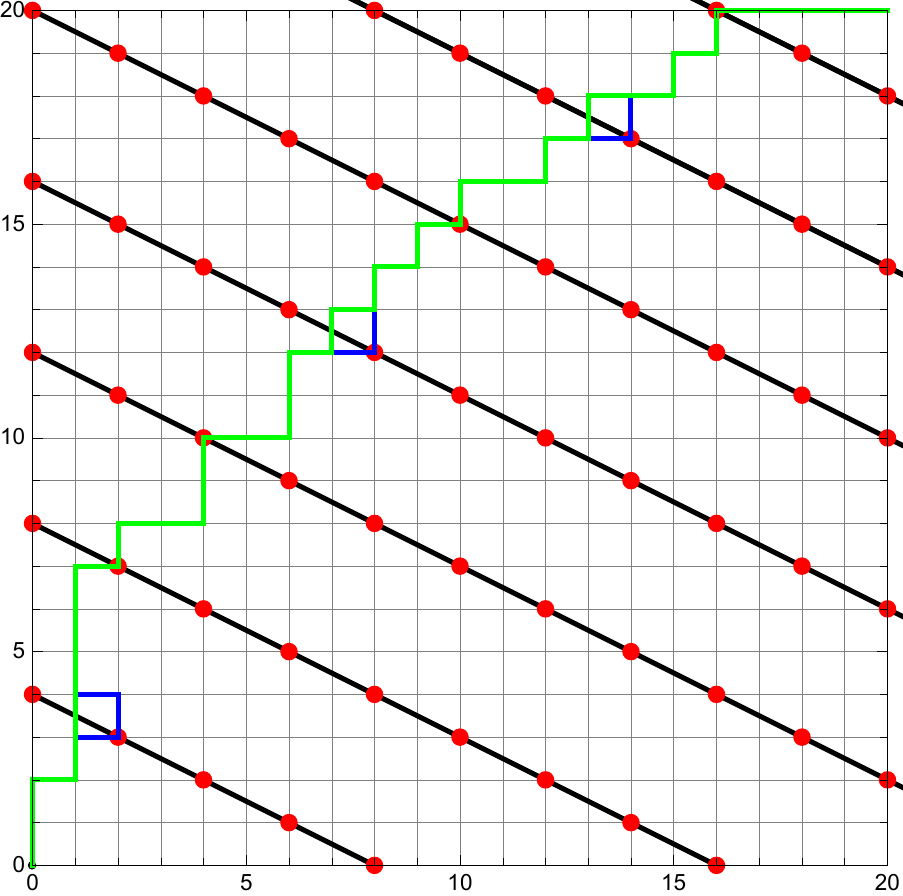}
    \caption{Coarse grid in red with parameters \(r=1/2, M =8, L=2 \). $\Gamma \in \path_N$ is shown in green. The modifications in blue result in a path $\Gamma' \in \pathp_N$: $\Gamma'$ follows the green path $\Gamma$ until it encounters a blue modification and takes that instead until it rejoins the green path.}
    \label{fig:coarse-grid}
\end{figure}

\begin{figure}[!htbp]
    \centering
    \includegraphics[width=0.5 \linewidth]{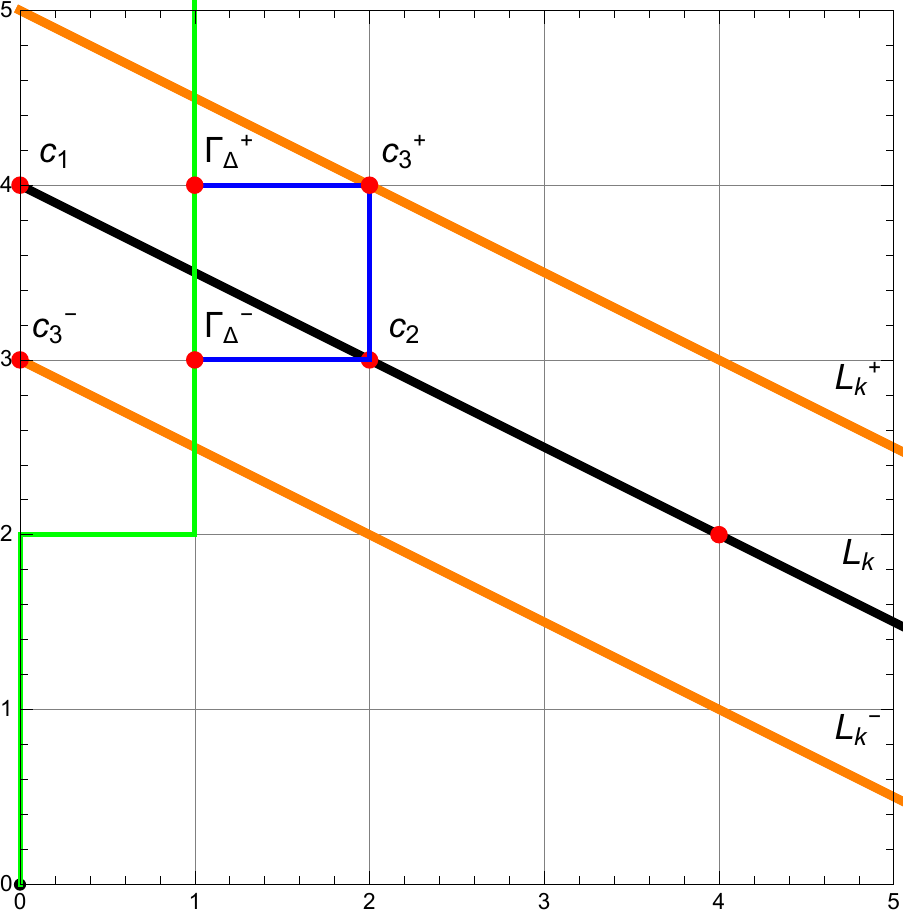}
    \caption{A zoomed-in view of the construction of $\Gamma'$ between the lines $\antidiaglines_k^-$ and $\antidiaglines_k^+$.}
    \label{fig:coarse-grid-wrong}
\end{figure}

For each \(\omega \in A_N\), by definition, there exists a path \(\Gamma \in \path_N\) with \(G(\Gamma)> N K\). This path can be modified between the lines $\antidiaglines_k^-$ and $\antidiaglines_k^+$ to obtain a path \(\Gamma' \in \pathp_N\). This construction is illustrated in Figures \ref{fig:coarse-grid} and \ref{fig:coarse-grid-wrong}: $\Gamma$ is modified so that $\Gamma'$ passes through a coarse grid point, and then takes a detour in the free-zone to rejoin $\Gamma$. We describe the construction precisely next. 

We define \(\Gamma'\) using an inductive construction over the lines \(\antidiaglines_k\). Let $\Gamma' = \Gamma$ until $\Gamma$ reaches $\antidiaglines_1^-$. Fix \(k \geq 1\) and assume $\Gamma$ and $\Gamma'$ have been defined up to the line $\antidiaglines_k^-$, and $\Gamma \cap \antidiaglines_k^- = \Gamma' \cap \antidiaglines_k^-$. Let $\Gamma_k, \Gamma_k^\pm$ be the points at which $\Gamma$ intersects $\antidiaglines_k$ and $\antidiaglines_k^\pm$ respectively. If $\Gamma_k \in C_g$, let $\Gamma' = \Gamma$ from $\antidiaglines_k^-$ to $\antidiaglines_{k+1}^-$. If not, we construct $\Gamma'$ between $\antidiaglines_k^-$ and $\antidiaglines_k^+$ as follows. Let $c_1,c_2$ be points in $C_g \cap \antidiaglines_k$ that are on either side of $\Gamma_k$, where $c_1$ has smaller $x$-coordinate. 

Let $c_3^- = c_2 - e_1  L r$. Consider the triangle formed by $c_1,c_2 \AND c_3^-$, and let $\Gamma_\triangle^-$ be the point at which $\Gamma$ intersects the sides $[c_3^-,c_2)$ or $(c_3^-,c_1)$. Let us consider the case where $\Gamma$ intersects $[c_3^-,c_2)$; the case where it intersects $(c_3^-,c_1)$ can be handled analogously. Let $\Gamma'$ follow $\Gamma$ from $\Gamma_k^-$ to $\Gamma_\triangle^-$ and then go horizontally in the $e_1$ direction to $c_2$. Similarly, we consider the triangle formed by the points $c_3^+ := c_1 + e_1  L r$, $c_1 \AND c_2$. Let $\Gamma_\triangle^+$ be the point where $\Gamma$ exits the triangle formed by $c_1,c_2, \AND c_3^+$. From $c_2$, $\Gamma'$ takes the shortest path along the sides $(c_1,c_3^+]$ and $(c_2,c_3^+)$ until it meets $\Gamma_\triangle^+$. Thereafter, $\Gamma'$ follows $\Gamma$ until it reaches $\antidiaglines_{k+1}^-$, and this completes the induction step. Finally, if $(N,Ns)$ falls between an $\antidiaglines_k^-$ and $\antidiaglines_k^+$, we just have $\Gamma'$ follow $\Gamma$.

The next lemma shows that replacing $\Gamma$ by $\Gamma'$ does not change the passage time substantially. Proposition \ref{prop:diagonal-lines} shows that \(\Gamma\) crosses a total of \( \frac{N (r+s)}{Mr} \) lines $\antidiaglines_k$ in \(\antidiaglines\). The modified path \(\Gamma'\) could incur a detour of length at most \(2L\) between each $\antidiaglines_k^- \AND \antidiaglines_k^+$, and so \(\Gamma'\) has length at most
\begin{align}
 |\Gamma ' |  \leq N(1+s) +2L \frac{N (r+s)}{Mr}.
 \label{eq:length of a gamma prime}
\end{align}
Since $\Gamma'$ does not coincide with $\Gamma$ only between $\antidiaglines_k^-$ and $\antidiaglines_k^+$, on the event $\good_N$, we have
\begin{equation}
     G(\Gamma') \geq G(\Gamma) - 2L \frac{N(r+s)}{Mr} (b_N - (-b_N)).
     \label{eq:estimate for passage time on modified path}
\end{equation}

Inspired by \eqref{eq:estimate for passage time on modified path}, we define the event \(\tilde A_N\) where
\begin{align} \tilde A_N =\left \{ \omega: \exists \Gamma' \in \pathp_N \text{ with } G(\Gamma')> N \left(K - 4L \frac{(r +s)}{Mr}b_N \right) \right\}.
\end{align}

\begin{lemma}\label{lemma: good path}
    We have \(A_N \cap \good_N \subseteq \tilde A_N \cap \good_N\) and thus \(P(A_N) \leq P( \tilde A_N) + c s N^2 e^{-\lambda b_N}\) for positive constants \(c,\lambda\) from \eqref{eq:bound on good complement using mgf}.
\end{lemma}

\begin{proof}
    For \( \w \in A_N \cap \good_N\), consider any \(\Gamma \in \path_N\). By the construction described above, there is a corresponding path \(\Gamma' \in \pathp_N\). By \eqref{eq:estimate for passage time on modified path}, it follows that  \(\omega \in \tilde A_N \cap \good_N\). Thus, 
    \begin{align*}
        P(A_N) 
        & = P(A_N \cap \good_N) + P (A_N \cap \good_N^c) \\
        & \leq P(\tilde A_N \cap \good_N )+ P( \good_N^c)\\
        & = P(\tilde A_N )+ c s N^2 e^{-\lambda b_N}.
    \end{align*}
\end{proof}

Let \(\alpha,\beta,\gamma >0\) with \(\alpha + \gamma< \beta\), let $b_N =  N^{\gamma}$, and let $L$ and $M$ be the smallest integers larger that \(\lfloor N^{\alpha} \rfloor \AND \lfloor N^{\beta} \rfloor\) respectively, such that $Lr$ and $Mr$ are integers.  From \eqref{eq:length of a gamma prime}, $n = |\Gamma'|$ satisfies the bound
    \begin{align}
        N(1+s) \leq n
        & \leq  N(1+s) + O( N^{1+ \alpha- \beta} ).
    \end{align}

Using a union bound over all paths in \(\pathp_N\) gives 
\begin{align}\label{eq:coarse-grid-inequality}
    & P( \tilde A_N) \\
    & \leq \big | \pathp_N \big| \max_{N(1+s) \leq n \leq N(1+s) + O( N^{1-\delta} )} P \left( \frac{1}{n} \sum_{i=1}^{n} X_i  \geq \frac{N}{n} \left( K - 4Lb_N \frac{ (s+r)}{Mr} \right) \right),\nonumber\\ 
    & \leq \big | \pathp_N \big| \sup_{N(1+s) \leq n } P \left( \frac{1}{n} \sum_{i=1}^{n} X_i  \geq \frac{K}{1+s} - O(N^{\alpha+\gamma-\beta}) \right),
\end{align}
where $X_i$ are iid with distribution $F$, and $\delta = 1 + \alpha - \beta < 1$. We estimate the term \(| \pathp_N \big|\) next.
\begin{prop}
    For fixed $s > 0$, $r \in \Q_{> 0}$, and $L$ and $M$ as described above,
    \[
        \frac{1}{N} \log \big | \pathp_N \big | = \frac{ s+r }{1+r} \log (4) + O(N^{-(\beta - \alpha)}).
    \]
    \label{prop:count for number of coarse-grained paths}
\end{prop}
\begin{proof}
Fix a path $\Gamma' \in \pathp_N$, and suppose $\Gamma'$ intersects the line $\antidiaglines_k^+$ at a point $x$. $\Gamma'$ lies inside a triangle $\triangle_{M'}$ with coordinates $x, x + M' e_1 \AND x + M'r e_2$, where $M' = M - L$, until it intersects the line $\antidiaglines_{k+1}$, which is the hypotenuse of \(\triangle_{M'}\). $\Gamma'$ must exit the triangle at a coarse-grained point in $C_g \cap \antidiaglines_{k+1}$. For the purpose of getting an upper bound on the number of up/right paths in $\triangle_{M'}$ that exit at a coarse-grained point, we may translate \(\triangle_{M'}\) to the origin ($x=0$) and replace $M'$ with $M$. Let $T_M$ be the number of such paths in the triangle $\triangle_{M}$. 

In $\triangle_M$, suppose $\Gamma'$ exits at a point $(x,y)\in C_g$ on the hypotenuse. Since paths must take a total of \(x+y = x+ r(M-x) \) steps from the origin to the hypotenuse, the number of possible paths is \( \displaystyle \binom{x+r(M-x)}{x}\). This binomial coefficient is maximized when \(x+r(M-x)=2x\), so there are at most \( \displaystyle \binom{ \frac{2Mr}{1+r}}{\frac{Mr}{1+r}}\) paths to \((x,y)\). Since the coarse grid points are a $\ell^2$ distance \(L \sqrt{1 + r^2}\) apart on the hypotenuse of $\triangle_M$, 
\[T_M \leq \frac{M }{L }\binom{ \frac{2Mr}{1+r}}{\frac{Mr}{1+r}}. \]

Once $\Gamma'$ has intersected $\antidiaglines_{k+1}$, in the free zone \(F_g^k\), $\Gamma'$ has length at most \(2L\) so there are at most \(4^{2L}\) such paths. This is a very crude bound, of course, but it makes no difference asymptotically at $N \to \infty$. Thus, we have accounted for all paths between $\antidiaglines_k^+$ to $\antidiaglines_{k+1}^+$. The number of paths from the origin to $\antidiaglines_1^+$ produces an identical bound.

A coarse-grained path crosses at most \(\frac{N(s+r)}{Mr}\) lines $\antidiaglines_k$, which gives the estimate
\begin{align*}
| \pathp_N | 
& \leq (T_M  4^{2L})^{\frac{N(s+r)}{Mr}} \leq \left(  \frac{M }{L }\binom{ \frac{2Mr}{1+r}}{\frac{Mr}{1+r}} 4^{2L} \right) ^{ \frac{ N (s+r)}{M r}}.
\end{align*}
Define the binary entropy function \(H(p)\) by
\begin{align}\label{eq:entropy function}
    H(p)&=-p\log (p)- (1-p) \log(1-p).
\end{align}
Using Stirling's formula to estimate the binomial coefficients gives
\begin{align*}
\log | \pathp_N | & = \frac{ N (s+r)}{M r} \left( \log \left(\frac{ M }{L}\right ) + \frac{2M r}{1+r} H(1/2) +O(1)+ 2L \log(4) \right)\\
& =  \frac{ N (s+r)}{1+r} \log (4) + O ( N^{1-\beta+\alpha}).
\end{align*}
Since \(\alpha < \beta < 1\) the error terms are all of order less than \(N\).
\end{proof}

Now we have all the tools needed to prove Theorem \ref{thm:rate function criterion for time-constant dominance}.
\begin{proof}[Proof of Theorem \ref{thm:rate function criterion for time-constant dominance}]
From \eqref{eq:coarse-grid-inequality}, 
\begin{align}
    P( \tilde A_N) 
    & \leq  \big | \pathp_N \big| \max_{N(1+s) \leq n} P \left( \frac{1}{n} \sum_{i=1}^{n} X_i  \geq \frac{K}{1+s} - o(1) \right), \label{eq: inequality union bound}
\end{align}
where the $o(1)$ term goes to $0$ as $N \to \infty$.

The standard large deviations estimate (Cram\'er's theorem) \citep{gs2001} applied to the i.i.d.\ sum $n^{-1} \sum_{i=1}^n X_i$ gives
\begin{align*}
\lim_{n \to \infty}  \frac{1}{n} \log P \left( \frac{1}{n} \sum_{i=1}^{n} X_i  \geq x \right) & =  - I\left(x \right).
\end{align*}
where \(I(x)\) is the large deviations rate function \eqref{eq:large deviation rate function definition}. Inserting this into \eqref{eq: inequality union bound} and using the estimate for $|\pathp_N|$ from \Propref{prop:count for number of coarse-grained paths}, we find that as $N \to \infty$,
\begin{equation}\label{eq:final inequality}
    \begin{aligned}
    \frac{1}{N(1+s) } \log P( A_N) \leq  \frac{\log (4) (s+r)}{(1+s)(1+r)} 
    & - I\left(\frac{K}{1+s} + o(1) \right) + o(1).
    \end{aligned}
\end{equation}
The continuity of $I$ and the choice of $r$ in \eqref{eq:choose r small enough} imply that the right hand side of \eqref{eq:final inequality} converges to zero as \(N\to \infty\) and thus \(g_F(s) \leq K\).
\end{proof}
\begin{figure}[!htpb]
    \centering
    \begin{subfigure}[b]{0.32\textwidth}
            \includegraphics[width=\textwidth]{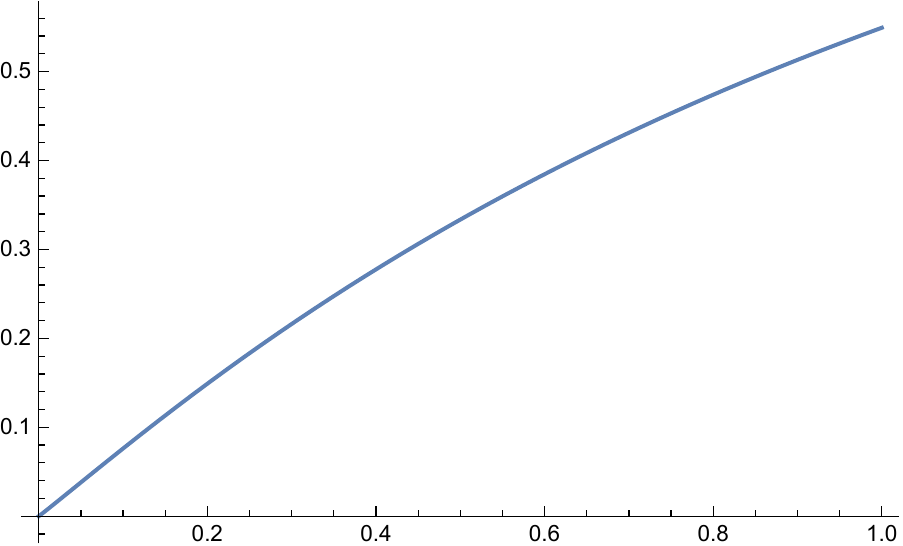}
            \caption{$p=.25$}
    \end{subfigure}
    \begin{subfigure}[b]{0.32\textwidth}
            \includegraphics[width=\textwidth]{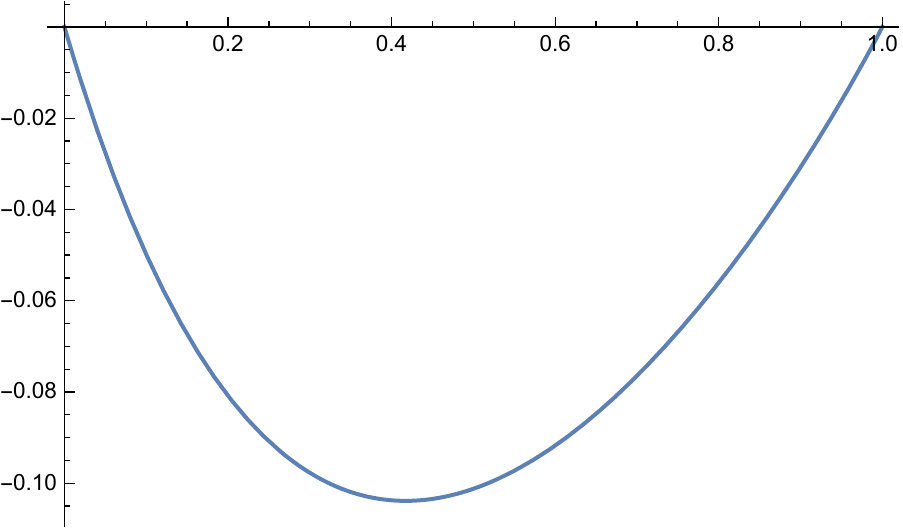}
            \caption{$p=.5$}
    \end{subfigure}
    \begin{subfigure}[b]{0.32\textwidth}
            \includegraphics[width=\textwidth]{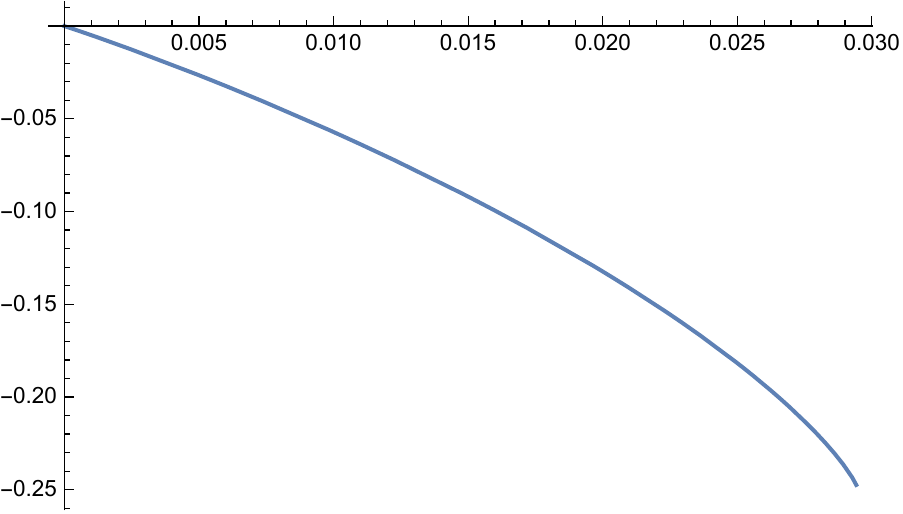}
            \caption{$p=0.75$}
    \end{subfigure}
    \caption{\(\phi(s)\) for different values of \(p\). Notice that the criterion in \Thmref{thm:main theorem criterion of negative covariance of adjacent busemann increments} for Bernoulli weights (see \Propref{prop:bernoulli weights}) is satisfied for $p = 1/2 \AND 3/4$ since $\phi(s) < 0$ for $0 < s < s^*$. The criterion is not satisfied for $p = 1/4$, and more generally (not shown here) for $p < 1/2$. Despite not satisfying the criterion, simulations (\Figref{fig:time constants}) show that $g_F \leq \gexp$ when $F = \Bernoulli(p)$ for $p < 1/2$.}
    \label{fig:coarse-grid-ber-25-r0}
\end{figure}

We used the following elementary count of the numbers lines in $\antidiaglines$ crossed by $\Gamma$ and $\Gamma'$ in the proof of \Thmref{thm:rate function criterion for time-constant dominance}.
\begin{prop}\label{prop:diagonal-lines}
    A path from \((0,0)\) to \(N(1,s)\) intersects a total of  \( { N (r+s)}/{(Mr)} + O(1) \) lines in \(\antidiaglines\).
\end{prop}
\begin{proof}
    An up/right path from \((0,0)\) to \(N(1,s)\) must cross every diagonal line in \(\antidiaglines\) once. Consider the rectangle $[0,N] \times [0,Ns]$. There are \(\lfloor N/M \rfloor\) lines in $\antidiaglines$ that intersect the eastern boundary of the box, and \(\lfloor (Ns)/(rM) \rfloor\) lines that intersect the northern. Therefore the path crosses a total of
    \( \lfloor N/M \rfloor + \lfloor (Ns)/(rM) \rfloor = \frac{N(r+s)}{rM} + O(1) \)
    diagonal lines.
\end{proof}

Next, we prepare to prove \Thmref{thm:main theorem criterion of negative covariance of adjacent busemann increments} by making a simplification that allows us reduce to the case \(0 < s < 1\). 
\begin{prop}
    Let $g_i$, $i=1,2$ be continuous, 1-homogeneous functions on $\R_{\geq 0}^2$ satisfying $g_i(x,y) = g_i(y,x) ~\forall x,y \in \R_{\geq 0}$. If \(g_1(1,s) \leq g_2(1,s)\) for all \(s\in (0,1)\) then $g_1(x) \leq g_2(x)$ for all \(x \in \R_{\geq 0}^2 \). 
    \label{prop:extending time constant inequality from 0-1 interval to entire first quadrant}
\end{prop}
\begin{proof}
    Suppose \(s>1\). Since \(1/s<1\) we have
    \begin{align*}
        g_{1}(1,1/s) \leq g_2(1,1/s).
    \end{align*}
    Using \(g_i(x,y)=g_i(y,x)\), we get
    \begin{align*}
        g_{1}(1/s,1) \leq g_2(1/s,1).
    \end{align*}
    Using 1-homogeneity, \(g_i(\lambda x,\lambda y)=\lambda g_i(x,y)\) for \(\lambda\geq 0\), and choosing \(\lambda=s\) we conclude 
    \begin{align*}
        g_{1}(1,s) \leq g_2(1,s).
    \end{align*}
    Thus, for arbitrary $x,y \in \R_{> 0}$, we have
    \begin{equation*}
        g_1(x,y) = x g_1(1,y/x) \leq x g_2 (1,y/x) = g_2 (x,y),
    \end{equation*}
    and by continuity, the result extends to $\R_{\geq 0}^2$.
\end{proof}

\begin{proof}[Proof of \Thmref{thm:main theorem criterion of negative covariance of adjacent busemann increments}]
    Since $\log(4) s /(1+s) < I(\gexp(1,s)/(1+s))$ for all $s \in (0,1)$, \Thmref{thm:rate function criterion for time-constant dominance} implies $g_F(1,s) \leq \gexp(1,s) ~\forall s \in (0,1)$. Since $g_F$ and $\gexp$ are continuous \citep[Theorem 2.4]{MR2094434}, 1-homogeneous functions satisfying $g_i(x,y) = g_i(y,x) ~\forall x,y \in \R_{\geq 0}^2$, $i \in \{F,\Exp\}$, \Propref{prop:extending time constant inequality from 0-1 interval to entire first quadrant} completes the proof. 
\end{proof}

\subsection{Failure of convex ordering}
This section contains the short proof of \Propref{prop:failure of convex ordering when matching moments}, that shows that two distinct random variables cannot be comparable in the convex ordering if they have equal first and second moments.
\begin{proof}
    We argue by contradiction. Suppose \(G \ll F\); then, by Theorem 2 in \citep{MR587206}, this is equivalent to the fact that there exists a coupling of $X$ and $Y$  such that \(E[ X | Y] \leq  Y ~\almostsurely\). Since \(E[X]=E[Y]\), the tower property of conditional expectation implies \(E[X| Y] = Y ~\almostsurely\). Applying the conditional Jensen's inequality, we have 
    \begin{align*}
        E[Y^2] & = E[ E[X|Y]^2]   \leq E[ E[X^2|Y] ]   = E[X^2].
    \end{align*}
    However, by assumption, \(E[X^2] = E[Y^2]\) and thus the inequality above must be an equality. Equality in the conditional Jensen's inequality holds if and only if $X = f(Y)$ a.s. Then, \(E[X|Y] = Y\)~a.s.\ implies that $X = Y$ a.s.
\end{proof}

\subsection{Bernoulli weights}
In this section we prove Proposition \ref{prop:bernoulli weights}: We demonstrate that Bernoulli weights with \(p \geq p^*\) ($p^* \approx 0.6504$) satisfy the criterion in Theorem \ref{thm:main theorem criterion of negative covariance of adjacent busemann increments}, and thus have negatively correlated adjacent Busemann increments. The criterion requires
\[
    \log(4) \frac{s}{1+s} - I\left( m + 2 \sigma \frac{\sqrt{s}}{1+s} \right) < 0 \quad \forall s \in (0,1),
\]
where $m = p$ and \(\sigma=\sqrt{p(1-p)}\). For convenience, we define
\begin{align*}
 u_s & = \frac{2 \sigma \sqrt{s}}{1+s}.
\end{align*}

 \begin{prop} \label{prop: bernoulli inequality}
    Let 
    \[\phi(s)=\frac{ \log(4) s}{1+s} - I \left( p +u_s\right).\]
    Let \(0 < s^*(p) < 1\) be the unique solution of \(p+u_{s}=1\) for \(1> p>1/2\) (see \eqref{eq: s star}), and let $p^*$ be the unique solution of $\log(4) = \tfrac{1 + p}{1 + s^*(p)}$ in $(1/2,1)$.
   If \(X\sim \text{Ber}(p)\) with \(p^*<p<1\) then 
   \[\phi(s) \text{ is } \begin{cases} < 0 & \qquad \text{if} \qquad 0< s <s^*(p) \\ = 0 & \qquad \text{if} \qquad s=0 \\ -\infty & \qquad \text{if} \qquad s \geq s^* \end{cases}.\]
    \end{prop}
    
    \begin{proof}
   The rate function \(I(x)\) for Bernoulli weights is 
   \begin{align*}
       I(x) & =   
           \begin{cases} 
               x \log \left(\frac{x}{p} \right) + (1-x) \log \left( \frac{ 1-x}{1-p} \right) & 0 < x  < 1 \\
               -\infty & \text{otherwise}
           \end{cases}
   \end{align*}
   So in our case, we only have to consider $s \in (0,1)$ such that $0 < p + u_s < 1$. For ease of computation we write \(\phi(s)\) in terms of the entropy function defined in \eqref{eq:entropy function}.
   We have
   \begin{align*}
      \phi(s) &=\frac{ \log(4) s}{1+s} - (p+u_s) \log \left(\frac{p+u_s}{p} \right) - (1-(p+u_s)) \log \left( \frac{ 1-(p+u_s)}{1-p} \right)\\
       & =  \frac{ \log(4) s}{1+s} +H(p+u_s)+u_s \log\left( \frac{p}{1-p} \right) -H(p).
   \end{align*}
   Note that \(\phi(0)=-H(p)+H(p)=0.\) We show \(\phi'(s)<0\) for \(0<s<s^*\). Using \(H'(x) = \log \left( \frac{1-x}{x} \right)\), we get 
   \begin{align}
       \phi'(s)&=\frac{\log(4)}{(1+s)^2}+u'_s\left( \log\left( \frac{p}{1-p} \right)  +\log\left( \frac{1-(p+u_s)}{p+u_s} \right) \right) \\
       & = u_s' \left(\frac{ \log(4)\sqrt{s}}{\sigma(1-s)}-\log\left(1+\frac{u_s}{p} \right) +\log\left( 1-\frac{u_s}{1-p} \right) \right). \label{eq:phi prime}
   \end{align}
   
   Now we focus on the term inside the parentheses in \eqref{eq:phi prime}, 
   \begin{align*}
       \psi(s):=\frac{ \log(4)\sqrt{s}}{\sigma(1-s)}-\log\left(1+\frac{u_s}{p} \right) +\log\left( 1-\frac{u_s}{1-p} \right),
   \end{align*}
   and bound this from above.
   Note that the condition for the rate function to be finite is \( p+u_s< 1\) which implies 
   \begin{align*}
   \frac{u_s}{p}< \frac{1-p}{p}<1, \text{ and } \frac{u_s}{1-p}< \frac{1-p}{1-p}<1,
   \end{align*}
   since \(1/2<p<1\). 
   We now use the following two elementary inequalities for the logarithm
   \begin{align*}
   \log(1+x)>  \frac{x}{2}, \quad \log(1-x) <-x  \quad \forall ~ 0<x<1.
   \end{align*}
   Inserting these inequalities gives
   \begin{align*}
       \psi(s) &\leq \frac{ \log(4)\sqrt{s}}{\sigma(1-s)}-\frac{u_s}{2p}  -\frac{u_s}{1-p}  = \frac{ \log(4)\sqrt{s}}{\sigma(1-s)} -\frac{(1+p)\sqrt{s}}{\sigma (1+s)}\\
       & = \frac{\sqrt{s}}{\sigma}\left( \frac{\log(4)}{1-s} - \frac{1+p}{1+s} \right).
   \end{align*}
   Plugging this back into \eqref{eq:phi prime} gives 
   \begin{align*}
       \phi'(s) & \leq \frac{1}{(1+s)^3} \left((1+s)\log(4) - (1+p)(1-s) \right).
   \end{align*}
   Integrating this inequality from $0$ to $s$, we get 
   \begin{align}
       \phi(s) & \leq \frac{s ((1+s)\log (4)-(1+p))}{(s+1)^2}
   \end{align}
   using the fact that \(\phi(0)=0\). Thus, we find \(\phi(s)<0 \) if \(\log 4< \frac{1+p}{1+s}\). 

    Solving \(p+g_{s}=1\) for $s$ in terms of $p$ gives 
   \begin{align}
       s^*(p)=\frac{1-3 p + 2 \sqrt{p(2p-1)}}{p-1}\label{eq: s star}.
   \end{align}
   Since \(0< s< s^*(p)\), it is enough to show \(\log 4 < \frac{1+p}{1+s^*(p)}\). Since \(s^*(p)\) is a strictly decreasing function for \(p>1/2\) (see \eqref{eq: s star}), if \( \log(4)< \frac{1+p^*}{1+s^*(p^*)}\) for some $p^*$, then $\phi(s) < 0 \quad \forall s \in (0,s^*(p))$ for any $p > p^*$. Let $p^*$ be the solution of \( \log(4)= \frac{1+p}{1+s^*(p)}\). Solving this numerically for \(p^*\), we find that \(p^* \approx .6504\).
   
   Finally we verify that \(s^*(p)\) is a decreasing function of \(p\) for \(p>1/2\) by computing its derivative:  
   \begin{align}
   \frac{d s^*}{dp\hphantom{^*}} & =  \frac{-3 p+2 \sqrt{p (2 p-1)}+1}{(p-1)^2 \sqrt{p (2 p-1)}}, \nonumber\\
   & \leq \frac{ 1-3p +  3p-1}{(p-1)^2 \sqrt{p (2 p-1)}} = 0, \label{eq: ds}
   \end{align}
   where we have applied the AM-GM inequality since \(p> 1/2\).
    \end{proof}

    \begin{proof}[Proof of \Corref{cor:criterion for general shifted bernoulli}]
        Let \(Y=(b-a)X+a\) where \(X\) is $\Bernoulli(p)$. Using \eqref{eq:large deviation rate function definition}, it is easy to see that the rate functions \(I_X\) and \(I_Y\) of \(X\) and \(Y\) satisfy 
        \begin{align*}
            I_Y(s) & = I_X \left( \frac{s-a}{b-a} \right).
        \end{align*}

        Let $m_X,m_Y,\sigma_X,\sigma_Y$ be the means and variances of $X$ and $Y$. Since \(m_Y = m_X (b-a) +a\) and \(\sigma_Y =(b-a) \sigma_X\),
        \begin{align*}
            I_Y\left( m_Y+\frac{2 \sigma_Y\sqrt{s}}{1+s} \right)
            & = I_X \left(  m_X + \frac{2 \sigma_X\sqrt{s}}{1+s} \right),
        \end{align*}
        and Proposition \ref{prop: bernoulli inequality} applies. 
    \end{proof}

\appendix
\section{Existence of Busemann functions}
\label{sec:appendix diff of limit shape}
For any $x \in \directions$, let $H_x = \{ \lambda x \colon \lambda > 0 \}$ be the line beginning at the origin that passes through $x$. Since the time-constant $g$ is $1$-homogeneous, $g'(x)$, if it exists, is constant along any $H_x$. Let $L_x$ be a tangent line of the limit-shape $\limitshape$ that intersects $H_x$; then there exist $x_L,x_R \in (0,1)$ that are the smallest and largest numbers, not necessarily distinct, such that the lines $H_{(x_L,1-x_L)}$ and $H_{(x_R,1-x_R)}$ intersect $L_x \cap \limitshape$. In \lpp{}, when $g$ is differentiable at both endpoints $(x_L,1-x_L)$ and $(x_R,1-x_R)$ \citep{MR3704768}, the limit in \eqref{eq:busemann-fn-definition} is known to exist and produce Busemann functions satisfying properties $(1)-(5)$ in \Defref{def:pre Busemann functions}. 
In \fpp{}, similar results have been proved in \citep{MR3152744}. More recently, in \fpp{}, \citet{2016arXiv160902447A} removed the differentiability requirements at the points \((x_L,1-x_L)\) and \((x_R,1-x_R)\), and showed that there is a unique Busemann function associated with each tangent line of the limit-shape. It is expected that their techniques can be extended to prove a similar result for \lpp{}.

Since the time-constant of \lpp{} is concave, it is not differentiable, in the worst case, on a countable set of points in $\directions$. However, the only available result about differentiability is where the minimum element of the support of the weights is an atom having probability larger than the critical probability for directed percolation. Here, we know that the boundary of the limit-shape is a straight line between two angles $\theta_1 < \theta_2$ that are symmetric about $(1/2,1/2)$ in the positive quadrant ---the so-called percolation cone--- and that the limit-shape is differentiable at the end points $\{\theta_1,\theta_2\}$ \citep{MR3535900}. These differentiability results in \lpp{} are based on earlier work in \fpp{} \citep{durrett_shape_1981,marchand_strict_2002,auffinger_differentiability_2013}.  

\section{Busemann correlations and the KPZ relationship}
\label{sec:appendix busemann correlations and kpz relationship}
For some $u \in \derivatives$, suppose $B^{u}$ is some Busemann function satisfying \eqref{eq:busemann-fn-definition} and the conditions in \Defref{def:pre Busemann functions}. It is expected that
\begin{equation}
    B^{u}(0,N v) = \grad g(u) \cdot Nv + \Theta(N^{1/2}),
    \label{eq:fluctuations of busemann functions}
\end{equation}
for any $v$ that is not parallel to $x$, and $\Theta(x)$ means that the quantity is bounded above and below by a constant times $x$. In the case of exponential or geometric weights, this is known to be true since Busemann increments are i.i.d.\ exponentials or geometrics respectively on any down/right lattice path, and the CLT implies their diffusive behavior \citep{MR2268539}. In \fpp{}, under various unproven hypotheses on the time-constant and the tail-behavior of the passage-time, Gangopadhyay \citep{gangopadhyay_fluctuations_2020} proves a result that suggests \eqref{eq:fluctuations of busemann functions} as well.  

Assuming \eqref{eq:fluctuations of busemann functions}, the following heuristic argument due to Newman, Alexander and others shows that $2 \chi = \xi$ in $d = 2$.  From Johansson's theorem \eqref{eq:gue limit for last passage time}, it follows that it is not unreasonable to expect that in general, for any $u \in \directions$,
\begin{equation}
    G(0,Nu) = N g(u) + \Theta(N^{\chi}).
    \label{eq:order of magnitude of passage-time fluctuations}
\end{equation}

The corrector or recovery property \eqref{eq:busemann function dynamic programming or recovery property} can be used to recover geodesics from Busemann functions. Let 
\begin{equation}
    \alpha(x) = \operatorname*{argmin}_{y \in \{ e_1,e_2\}} B^u(x,x + y)  
    \label{eq:arrow definition}
\end{equation}
be the arrow at $x$. In case of a tie in \eqref{eq:arrow definition}, we may always assume that $\alpha(x) = e_1$. Given any $x \in \Z^2$, we can form an up/right lattice path as follows: let $X_0 = x$, and $X_n = X_{n-1} + \alpha(X_{n-1})$ for $n \geq 1$. It can be shown such paths formed by following arrows always produce geodesics, and these are called \textbf{Busemann geodesics} \citep[eq. (2.14)]{MR3704769}.

The Licea-Newman argument \citep{MR1387641,MR3152744} shows that Busemann geodesics from any two points coalesce almost surely\ \citep[Theorem 4.5]{MR3704769}. Since the geodesics from $0$ to $Nx$ and $N^\xi v$ to $Nx$ fluctuate on the $N^{\xi}$ scale, it is expected that there is a random tight constant $a_N$ such that the geodesics from $0$ and $N^{\xi}v$ would have merged after $a_N N x$ steps. Indeed, this is known to be true in the exponential and geometric cases \cite{MR4002528}. Then, from \eqref{eq:busemann-fn-definition} 
\[
    B^{x}(0,N^{\xi}v) = G(0,a_N N x) - G(0,a_N N x),
\]
for some large enough $N$. Inserting \eqref{eq:order of magnitude of passage-time fluctuations} and \eqref{eq:fluctuations of busemann functions} into the above, we get
\begin{align*}
    \grad g(x) \cdot N^{\xi} v + \Theta(N^{\xi/2}) 
    & = N g(a_N x) - N g(a_N x - N^{\xi - 1} v) + \Theta(N^{\chi}) \\
    & \approx \grad g(a_N x) \cdot N^{\xi} v + \Theta(N^\chi).
\end{align*}
By the 1-homogeneity of $g(u)$, $\grad g(a_N x) = \grad g(x)$, and thus $\xi = 2\chi$.

An initial step towards proving \eqref{eq:fluctuations of busemann functions} is to show
\begin{equation}
    \Var(B^{x}(0,Nv)) \leq O(N).
    \label{eq:order N bound for variance of busemann}
\end{equation}
If the covariance inequalities in \eqref{eq:negative correlation of busemann functions conjecture-0} and \eqref{eq:negative correlation of busemann functions conjecture} hold, we have
\begin{align}
    \Var(B^{x}(0,Nv)) 
    & = \Var\left( \sum_{i=0}^{N-1} B^{x}(v_j, v_{j+1}) \right) \nonumber\\
    & = \sum_{i=0}^{N-1} \Var\left( B^{x}(v_i, v_{i+1})\right) + 2\sum_{i < j} \Cov(B^{x}(v_i, v_{i+1}),B^{x}(v_j, v_{j+1}) ) \nonumber\\
    & \leq \Var(B^{x}(0,v)) N, \label{eq:variance of a busemann increment}
\end{align}
where the first equality follows from additivity and the last uses the stationarity of Busemann functions. If only \eqref{eq:order N bound for variance of busemann} is available, the heuristic gives $\xi \geq 2 \chi$.

\subsection*{Acknowledgments}   
We would like to thank T.\ Sepp{\"a}l{\"a}inen for sharing a short note that connects zero correlations of adjacent Busemann increments and the universal exponential limit-shape function \eqref{eq:exponential-lpp-limit-shape}; M.\ Damron for suggesting the use of coarse graining to improve the union bound in \eqref{eq:naive set}; C.\,Janjigian for pointing out that pre-Busemann functions always exist; and M.\,Hegde for sharing his simulation of the $\chi^2$ distribution and showing us a better way to present our simulations. A.\,Krishnan would like to acknowledge support from a Simons Collaboration Grant 638966.
\printbibliography


\end{document}